\documentclass{amsart}

\usepackage{a4wide,amsmath,amssymb}
\usepackage[toc,page]{appendix}

\usepackage{url}
\usepackage{xcolor}

\newtheorem{thm}{Theorem}[section]
\newtheorem{lemma}[thm]{Lemma}
\newtheorem{prop}[thm]{Proposition}
\newtheorem{cor}[thm]{Corollary}
\theoremstyle{remark}
\newtheorem{rem}[thm]{Remark}

\newtheorem{ex}[thm]{Example}

\theoremstyle{definition}
\newtheorem{defn}[thm]{Definition}
\newtheoremstyle{Claim}{}{}{\itshape}{}{\itshape\bfseries}{:}{ }{#1}
\theoremstyle{Claim}

\newcommand{\R}{\mathbb{R}}

\newcommand{\N}{\mathbb{N}}

\newcommand{\He}{\mathbb{H}}
\newcommand{\X}{\mathcal{X}}

\newcommand{\eps}{\varepsilon}

\newcommand{\al}{\alpha}

\newcommand{\g}{\gamma}
\newcommand{\G}{\mathbb{G}}

\DeclareMathOperator{\LSC}{LSC}
\DeclareMathOperator{\USC}{USC}
\newcommand{\Sym}{\mathcal{S}}
\theoremstyle{plain}

\def\sideremark#1{\ifvmode\leavevmode\fi\vadjust{
\vbox to0pt{\hbox to 0pt{\hskip\hsize\hskip1em
\vbox{\hsize3cm\tiny\raggedright\pretolerance10000
\noindent #1\hfill}\hss}\vbox to8pt{\vfil}\vss}}}

\begin{document}

\title[Liouville theorems]{Liouville results for fully nonlinear equations modeled on H\"ormander vector fields: II. Carnot groups and Grushin geometries}
\author{Martino Bardi}
\author{Alessandro Goffi} 

\date{\today}
 \thanks{
 The authors are members of the Gruppo Nazionale per l'Analisi Matematica, la Probabilit\`a e le loro Applicazioni (GNAMPA) of the Istituto Nazionale di Alta Matematica (INdAM). 
 { This research was carried out while Alessandro Goffi was Ph.D. fellow at Gran Sasso Science Institute, L'Aquila, and Postdoctoral researcher at the University of Padova, and some of the results are part of his Ph.D. thesis. 
 {We wish to thank Prof. L. D'Ambrosio for pointing out reference \cite{DM19}.}}
 }
\address{Department of Mathematics ``T. Levi-Civita", University of Padova, Via Trieste 63, 35121 Padova, Italy} \email{bardi@math.unipd.it}
\address{Department of Mathematics ``T. Levi-Civita", University of Padova, Via Trieste 63, 35121 Padova, Italy} \email{alessandro.goffi@unipd.it}

\maketitle
\begin{abstract}
The paper treats second order fully nonlinear degenerate elliptic equations having a family of subunit vector fields satisfying a full-rank bracket condition. It studies Liouville properties for viscosity sub- and supersolutions in the whole space, namely, that under a suitable bound at infinity from above and, respectively, from below, they must be constants. In a previous paper we proved an abstract result and discussed operators on the Heisenberg group. Here we consider various families of vector fields: the generators of a Carnot group, with more precise results for those of step 2, in particular H-type groups and free Carnot groups, the Grushin and the { Heisenberg-Greiner} vector fields. All these cases are relevant in sub-Riemannian geometry and have in common the existence of a homogeneous norm that we use for building Lyapunov-like functions for each operator. We give explicit sufficient {conditions} on the size and sign of the first and { zero}-th order terms in the equations and discuss their optimality. 
{We also 
outline some applications of such results to the problem of 
 ergodicity of multidimensional degenerate diffusion processes in the whole space.}
\end{abstract}

\section{Introduction}
\label{intro}
This paper continues our analysis initiated in \cite{BG_lio1} on one-side Liouville properties for entire viscosity sub- and supersolutions of fully nonlinear subelliptic PDEs of the form
 \begin{equation}\label{fullyintro}
G(x,u,D_\X u,(D^2_\X u)^*)=0\text{ in }\R^d\ ,
\end{equation}
where $\X=\{X_1,...,X_m\}$ is a family of H\"ormander vector fields, { $u:\R^d\to\R$}, $D_\X u:\R^d\to\R^m$ and $(D^2_\X u)^*:\R^d\to\mathrm{Sym}_m$, $m\leq d$, are respectively the horizontal gradient and the symmetrized horizontal Hessian of the unknown function $u$.   
Our abstract result in  \cite{BG_lio1} considers  operators $G$ satisfying some general structural assumptions that we recall precisely in Section \ref{abs}. 
 We suppose the existence of an {\em exhaustion function} $w$ { \cite{Gry1,MPessoa1,MPessoa2}}, i.e., such that $\lim_{|x|\to\infty}w(x)=+\infty$, that we call 
 a {\em Lyapunov function} if it is a viscosity supersolution of \eqref{fullyintro} outside a compact set, 
  and impose to $u$ the bound from above for large $|x|$
\begin{equation}\label{above}
\limsup_{|x|\to\infty}\frac{u(x)}{w(x)}\leq 0\ .
\end{equation}
We call Liouville property for subsolutions of \eqref{fullyintro} the following:
\begin{multline}
\label{LP1}
\tag{LP1}
\text{ if $u\in\USC(\R^d)$ is a viscosity subsolution to \eqref{fullyintro}}\\
\text{satisfying \eqref{above}
for a Lyapunov function $w$, then $u$ is constant}. 
\end{multline}
Symmetrically, one can formulate a Liouville property for $\LSC(\R^d)$ supersolutions $v$ to \eqref{fullyintro} by assuming the existence of a function $W$ viscosity subsolution 
 to \eqref{fullyintro} outside a compact set, such that $\lim_{|x|\to\infty}W(x)=-\infty$, that we call
 a {
  \em negative Lyapunov function}, quantifying the bound from below at infinity via
\begin{equation}\label{below}
{ \limsup_{|x|\to\infty}}\frac{v(x)}{W(x)}\leq 0 \, .
\end{equation}
We call now Liouville property for supersolutions of \eqref{fullyintro} the following:
\begin{multline}
\label{LP2}
\tag{LP2}
\text{ if $v\in\LSC(\R^d)$ is a viscosity supersolution to \eqref{fullyintro}}\\
\text{satisfying \eqref{below} for a negative Lyapunov function $W$, then $v$ is constant}. 
\end{multline}
Therefore the validity of such properties boils down to the construction of suitable Lyapunov functions. For linear equations this is known in the literature as the {\em Khas'minskii test} and it is deeply connected with the recurrence and ergodicity properties of the associated diffusion process: see, e.g.,  the nice survey \cite{Gry1} on the extensions to Riemannian manifolds, \cite{MV} for quasilinear operators like the $p$-laplacian { and the recent papers \cite{MPessoa1,MPessoa2}{, the monograph \cite{BMPR}} along with the references therein}.

Such Lyapunov functions were first built  for fully nonlinear uniformly elliptic equations in \cite{BC} by exploiting the comparison with convex or concave operators, in particular the Pucci's extremal operators $\mathcal{M}_{\lambda,\Lambda}^-(D^2u), \mathcal{M}_{\lambda,\Lambda}^+(D^2u)$. In our previous paper  \cite{BG_lio1}  we built Lyapunov functions for degenerate equations with a similar structure on the Heisenberg group $\He^d$ by means of the norm $\rho$ homogeneous with respect to the dilations of the group, and we checked that the conditions for the Liouville properties were sharp by computing $\mathcal{M}_{\lambda,\Lambda}^\pm((D^2_{\He^d} f(\rho))^*)$ for suitable $f$. 

In the present paper we build Lyapunov functions, and therefore get Liouville properties,  for several other choices of 
 vector fields $\X$ that are of interest in sub-Riemannian geometry. We begin with homogeneous Carnot groups, for which we refer to the comprehensive monograph  \cite{BLU}. In this generality we discuss in particular a Liouville comparison principle inspired by \cite{KurtaKawohl}, the failure of Liouville properties for linear equations, { together with} an estimate of the distance at infinity of a supersolution from {the constant of the Liouville property}, inspired by \cite{Kurta, Veron}
 {, see also \cite{DM19}}.

Next we turn to two classes of groups of step 2: H-type groups and free Carnot groups. 
Here we use  Lyapunov functions of the form $w=\log\rho$, where $\rho$ is the homogeneous norm of the group, for subsolutions of convex equations or supersolutions of concave ones. We also discuss  the optimality of our conditions for getting (LP1) and (LP2) and consider equations involving the horizontal Hessian $D^2_\X u$ and the Euclidean gradient $Du$.

An example of result we get for {a model  H-type} group on $\R^7$ {introduced }in \cite[p.687]{BLU} is the following: 
 \eqref{LP1} holds when
\begin{equation}
\label{Gbelow}
G(x,u,D_\X u,(D^2_\X u)^*)\geq \mathcal{M}_{\lambda,\Lambda}^-((D^2_\X u)^*)+\inf_{\alpha\in A}\{c^\alpha(x)u-b^\alpha(x)\cdot D_\X u\} \, ,
\end{equation}
with $c^\alpha\geq0$, $c^\alpha, b^\alpha$ uniformly locally Lipschitz, and 
\begin{equation}
\label{recur}
\sup_{\alpha}\left\{b^\alpha(x)\cdot D_\X\rho(x)\frac{\rho^3}{|x_H|^2}-c^\alpha(x) \frac{\rho^4\log\rho}{|x_H|^2} \right\}\leq \lambda -\Lambda(Q-1)\ \text{for }|x|\geq R\ ,
\end{equation}
where $x_H$ is the horizontal part of $x$ and $Q$ is the homogeneous dimension of the structure 
\cite{BLU}. Since $D_\X\rho$ has explicit polynomial  components and $|D_\X\rho|=|x_H|/\rho$, one can easily check the last condition.
For instance, \eqref{recur} holds for $c^\alpha\geq c_o>0$ and $b^\alpha$ bounded, whereas for $c^\alpha=0$ it becomes a {\em recurrence condition} on the drift, saying that $b^\alpha\cdot D_\X\rho$ must be negative and large enough in norm for large $|x|$. Under the assumptions  \eqref{Gbelow} and  \eqref{recur} the Liouville comparison principle mentioned before states that a subsolution $u$ and a supersolution $v$ of \eqref{fullyintro}, such that $\limsup_{|x|\to\infty}\frac{u -v}{\log\rho}\leq 0$,  coincide up to a constant. Note that this is equivalent to the Liouville property for linear equations, 
but not in the nonlinear case, and it appears to be new  for fully nonlinear  equations even in the uniformly elliptic  Euclidean setting.

In the second part of the paper we consider geometries that are not related to a group structure. The first is the classical one associated to the Grushin vector fields, in the plane as well as in the generalized version for arbitrary dimension.  
Also here there is  a suitable norm $\rho$ associated to the fundamental solution of the sub-Laplacian \cite{DaLucente} and we consider operators that can be compared with convex or concave ones, as in \eqref {Gbelow}. The Liouville-type results that we find assume conditions on the data of the same form as \eqref{recur}, they are sharp in any dimension for quasi-linear equations  and in the plane for fully nonlinear operators. 

The last sub-Riemannian structure taken into account is generated by the Heisenberg-Greiner vector fields, see \cite{BGG1, BGG2, BieskeF}, that is intermediate between Heisenberg and Grushin geometries. Again, a gauge norm $\rho$ allows us to find a Lyapunov function.

We refer to our companion paper \cite{BG_lio1} for a 
general  introduction to Liouville properties for (degenerate) elliptic equations and their motivations and applications. The recent article \cite{CirantGoffi} discusses these properties for equations with superlinear growth in the gradient $Du$ and presents several open problems and a very large bibliography. 

{Let us emphasize that general Liouville-type results for fully nonlinear subelliptic equations cannot be deduced from 
} (invariant) Harnack inequalities as in the classical uniformly elliptic setting \cite{CC}, because  a building block as the ABP maximum principle is still unknown in this framework. Some recent advances in this direction can be found in \cite{Tralli} and the references therein. 

{The main potential applications of our 
Liouville 
 properties concern 
various forms  of ergodicity of 
 controlled 
  diffusion processes such as 
\[
dX_t=b(X_t, \al_t)dt+\sigma(X_t, \al_t) dB_t\ ,X_0=x\in\R^d,
\]
where $\al_t$ is a control function, $B_t$ a Brownian motion, and $\sigma$ a matrix whose entries are the coefficients of H\"ormander vector fields.
This is related to the large-time behavior for 
 degenerate parabolic Hamilton-Jacobi-Bellman equations. 
 Sufficient conditions for the ergodicity of such 
  processes on compact state spaces have been thoroughly discussed in the literature, cfr. \cite{AB} and the references therein, whilst no general criteria have been systematically explored 
  when the process is posed on the whole space. For uncontrolled processes 
   the analysis by PDE methods started in the work \cite{LionsMusiela} 
 and 
  continued in the Lions' lectures \cite{LionsCourse}, which inspired the recent work \cite{MMT}. At the end of the paper we show how the Lyapunov functions constructed in the previous sections can be used to prove asymptotic properties for linear degenerate elliptic and parabolic equations.}
  

The paper is organized as follows. In Section \ref{abs} we recall some preliminaries and the main abstract result from \cite{BG_lio1}, and we apply it to an abstract Liouville comparison principle.
Section \ref{sec_gen} first recalls some basic facts on Carnot groups and their sub-Laplacians, then gives a Liouville theorem for operators concave in the Hessian and convex in the gradient (or viceversa). It continues with examples of non-constant sub- and supersolutions of linear equations and with a theorem on the asymptotic behaviour at infinity of non-constant semi-solutions, which appears to be new even in the Heisenberg group. Section \ref{sec_fully} contains our results about equations on Carnot groups of step 2. We analyse first nonlinear PDEs on H-type groups, including cases with dependence on the full Euclidean gradient. Then we turn to free Carnot groups, for which we find more precise results for the classical Pucci operators $\mathcal{P}^{\pm}_{\lambda}$ \cite{Pucci66} than for the usual $\mathcal{M}^\pm_{\lambda,\Lambda}$ as defined in {\cite{CC}}.
In Section \ref{sec_gru} we introduce the Grushin vector fields for which we prove Liouville properties for quasi-linear and fully nonlinear equations. 
Section \ref{sec_hg} presents the  Heisenberg-Greiner geometry and a Liouville theorem for  quasi-linear operators.
In Section \ref{sec_fin} we show that for equations of Ornstein-Uhlenbeck type without  terms of order 0 our recurrence condition on the drift is sharp for the Liouville property. 
{Section \ref{sec;erg} concludes the paper with some applications to asymptotic problems of ergodic type 
for linear degenerate elliptic and parabolic equations.}

\section{Abstract results}\label{abs}
\subsection{Liouville properties for H\"ormander vector fields}
Given a family $\X=\{X_1,...,X_m\}$ of smooth vector fields satisfying the H\"ormander's rank condition { (see below for the definition)}, we consider general fully nonlinear subelliptic equations of the form
\begin{equation}\label{fully1}
G(x,u,D_\X u,(D^2_\X u)^*)=0\text{ in }\R^d\ ,
\end{equation}
where $D_\X u=(X_1u,...,X_mu)$ and $(D^2_\X u)^*_{ij}=\frac{X_iX_ju+X_jX_iu}{2}$ stand respectively for the horizontal gradient and { the symmetrized Hessian}, $m\leq d$, and $G:\R^d\times\R\times\R^m\times\mathrm{Sym}_m\to\R$. 
The equation can be written in Euclidean coordinates by calling $\sigma=\sigma(x)$ the $d\times m$ matrix whose columns $\sigma^j$ have the coefficients of the fields $X_j$, and observing that
\[
D_\X u=\sigma^TDu\ ,(D^2_\X u)^*=\sigma^T(x)D^2u\sigma(x)+g(x,Du)\ ,
\]
where $g(x,Du)$ is the $m\times m$ matrix with entries
\[
g_{ij}(x,p)=\left(\frac{D\sigma^j(x)\sigma^i(x)+D\sigma^i(x)\sigma^j(x)}{2}\right)\cdot p\ .
\]

{In analogy with \cite{CC},} we will consider those operators that are \textit{uniformly subelliptic}, i.e. for $M,N\in\mathrm{Sym}_m$
\begin{equation}\label{unifsubell}
\mathcal{M}^-_{\lambda,\Lambda}(M-N)\leq G(x,r,p,M)-G(x,r,p,N)\leq \mathcal{M}^+_{\lambda,\Lambda}(M-N)
\end{equation}
where $\mathcal{M}^\pm_{\lambda,\Lambda}$ are the Pucci's extremal operators \cite{CC} over the symmetrized horizontal Hessian defined as
\[
\mathcal{M}^+_{\lambda,\Lambda}(M)=\sup\{-\mathrm{Tr}(AM)\ ,\lambda I_m\leq A\leq \Lambda I_m\}=-\lambda\sum_{e_i(M)>0}e_i-{\Lambda}\sum_{e_i(M)<0}e_i
\]
\[
\mathcal{M}^-_{\lambda,\Lambda}(M)=\inf\{-\mathrm{Tr}(AM)\ ,\lambda I_m\leq A\leq \Lambda I_m\}=-\Lambda\sum_{e_i(M)>0}e_i-\lambda\sum_{e_i(M)<0}e_i\ ,
\]
{where $e_i(M)$ are the eigenvalues of the matrix $M$}.
It is important to point out that condition \eqref{unifsubell} does not imply the uniform ellipticity in the classical sense, as in  
Caffarelli-Cabr\'e \cite{CC}, since we are considering $m$-dimensional symmetric matrices for an operator posed on the Euclidean space $\R^d$ with $m<d$. 
We can also express the uniform subellipticity condition \eqref{unifsubell} via Euclidean coordinates as
\begin{multline*}
\lambda\mathrm{Tr}(\sigma^T(x)P\sigma(x)+g(x,p))\leq G(x,r,\sigma^T(x)p,\sigma^T(x)N\sigma(x)+g(x,p))\\
-G(x,r,\sigma^T(x)p,\sigma^T(x)(P+N)\sigma(x)+g(x,p))\leq \Lambda\mathrm{Tr}(\sigma^T(x)P\sigma(x)+g(x,p))
\end{multline*}
for $P,N\in \mathrm{Sym}_d$, $P\geq0$, i.e., positive semidefinite. { Note 
 that the correction term $g(x,p)$ vanishes in many important cases, such as Carnot groups of step 2}. 
\begin{rem}\label{nondiv}
It is worth remarking that 
\[
G((D^2_\X u)^*)=-\mathrm{Tr}((D^2_\X u)^*)=\mathcal{M}^\pm_{1,1}((D^2_\X u)^*)=-\Delta_\X u=-\sum_{i=1}^mX_i^2 u
\]
and by the definition of {the} Pucci's extremal operators (which appear as a supremum or infimum of linear operators in non-divergence form), all the non-existence results stated for equations of the form $\mathcal{M}^-_{\lambda,\Lambda}((D^2_\X u)^*)\leq H(x,u,D_\X u)$ or $\mathcal{M}^+_{\lambda,\Lambda}((D^2_\X u)^*)\geq H(x,u,D_\X u)$ lead immediately to non-existence results for sub- and supersolutions to 
{the non-divergence structure equation}
\[
-\mathrm{Tr}(A(x)(D^2_\X u)^*)= H(x,u,D_\X u),
\]
provided the eigenvalues of $A\in\mathrm{Sym}_m$ lie in the interval $[\lambda,\Lambda]$. To our knowledge, most of the results we will present are new even for linear and quasi-linear PDEs driven merely by the sum of squares of H\"ormander vector fields.
\end{rem}
We now recall some abstract Liouville results obtained in \cite{BG_lio1} for fully nonlinear uniformly subelliptic equations \eqref{fully1}. When dealing with Liouville properties for viscosity subsolutions, we will consider those operators $G$ uniformly subelliptic such that
\begin{equation}\label{G>H}
G(x,r,p,0)\geq H_i(x,r,p)\ ,x\in\R^d\ ,p\in\R^m\ , r\in\R ,
\end{equation}
 with $H_i$ concave in $p$, so of the form
\begin{equation}\label{Hi}
H_i(x,r,p)=\inf_{\alpha\in A}\{c^\alpha(x)r-b^\alpha(x)\cdot p\}\ .
\end{equation}
Then it will be enough to treat equations of the form
\begin{equation}\label{P-}
\mathcal{M}^-_{\lambda,\Lambda}((D^2_\X u)^*)+H_i(x,u,D_\X u)=0\text{ in }\R^d .
\end{equation}
When dealing with Liouville properties for viscosity supersolutions, we will assume instead
\begin{equation}\label{G<H}
G(x,r,p,0)\leq H_s(x,r,p)\ ,x\in\R^d\ ,p\in\R^m\ , r\in\R ,
\end{equation}
with $H_s$ convex in $p$, i.e.,
\begin{equation}\label{Hs}
H_s(x,r,p)=\sup_{\alpha\in A}\{c^\alpha(x)r-b^\alpha(x)\cdot p\}\ .
\end{equation}
and therefore consider the equation
\begin{equation}\label{P+}
\mathcal{M}^+_{\lambda,\Lambda}((D^2_\X u)^*)+H_s(x,u,D_\X u)=0\text{ in }\R^d\ .
\end{equation}
As for the first and zero-th order coefficients, we will assume the following conditions: $b^\alpha:\R^d\to\R^m$ is locally Lipschitz in $x$ uniformly in $\alpha$, i.e., for all $R>0$ there exists $K_R>0$ such that
\begin{equation}\label{b}
\sup_{|x|,|y|\leq R}|b^\alpha(x)-b^\alpha(y)|\leq K_R|x-y|
\end{equation}
and
\begin{equation}\label{c}
c^\alpha(x)\geq0\text{ and continuous in }|x|\leq R\text{ uniformly in }\alpha\ .
\end{equation}
We recall that a family of vector fields 
satisfy the H\"ormander's rank condition if 
\begin{equation}\label{H}\tag{H}
\text{the vector fields are smooth and the Lie algebra generated by them has full rank $d$ at each point}.
\end{equation}
\begin{thm}
\label{main}
Assume that the vector fields $\X$ are $C^{1,1}$ and satisfy the H\"ormander condition \eqref{H}. Furthermore, suppose that $G$ is uniformly subelliptic.
\begin{itemize}
\item[(a)] Under the previous assumptions on $H_i$, \eqref{LP1} holds for \eqref{P-} provided that either $u\geq0$ or $c^\alpha(x)\equiv 0$.
\item[(b)] Under the previous assumptions on $H_s$, \eqref{LP2} holds for \eqref{P+} provided that $v\leq0$ or $c^\alpha(x)\equiv 0$.
\end{itemize}
As a consequence, \eqref{LP1} and \eqref{LP2} hold for the equation \eqref{fully1}.
\end{thm}
{This} result has been proved in \cite[Corollary 3.11 and 3.12]{BG_lio1} via the strong maximum principle recently obtained by the authors in \cite{BG} through the notion of generalized subunit vector fields for fully nonlinear operators. The result has extended prior Liouville properties obtained in \cite{BC} valid for fully nonlinear uniformly elliptic equations and some quasi-linear degenerate equations. 
Section 4.1 in \cite{BG_lio1} provides explicit sufficient conditions for the validity of the Liouville property 
for PDEs structured on
the Heisenberg group, 
but allows also a dependence on the Euclidean gradient, cf \cite[Section 4.3]{BG_lio1}. In the next sections, we will exhibit further sufficient conditions ensuring the validity of Liouville properties in more general structures. 

We finally remark that when the control set $A$ appearing in $H_i$ and $H_s$ is a singleton, and $\lambda=\Lambda=1$, our equations \eqref{P-}-\eqref{P+} reduce to the linear problem
\[
-\Delta_\X u+b(x)\cdot D_\X u+c(x)u=0
\]
and the results in Theorem \ref{main} apply to {it} 
 under the assumptions \eqref{b}-\eqref{c} on $b,c$ 
if there exist the Lyapunov functions $w$ and $W$. To our knowledge, the results we obtain here are new even for {such} 
 linear equations.

\subsection{A Liouville comparison principle.}
A consequence of Theorem \ref{main} is the following 
Liouville property 
expressed in the form of a comparison principle.
\begin{cor}      
\label{lcpl} 
Assume  $\X$ are $C^{1,1}$ and satisfy  \eqref{H}, $c^\alpha\equiv 0$, and there exists a Lyapunov function $w$ for \eqref{P-} (resp., a negative Lyapunov function $W$ for \eqref{P+}). Let $u$ and $v$ be  a viscosity sub- and supersolution to \eqref{P-} (respectively, to \eqref{P+})
such that $\limsup_{|x|\to\infty}\frac{u(x) - v(x)}{w(x)}\leq 0$ (resp.,  $\liminf_{|x|\to\infty}\frac{v(x)-u(x)}{W(x)}\leq 0$).
Then $u\equiv v$ in $\R^d$ up to a constant.
\end{cor}
\begin{proof}
It is not hard to check by the arguments in \cite{BC, BG_lio1} that in the first case $u-v$ is a viscosity subsolution of \eqref{P-}. Since it satisfies \eqref{above}, by Theorem \ref{main} (a) it is a constant.

Similarly, in the second case $v-u$ is a supersolution of \eqref{P+} and it satisfies \eqref{below}, so Theorem \ref{main} (b) says that it is a constant.
\end{proof}
\begin{rem}
Note that the growth conditions at infinity on $u-v$ are both satisfied if $u\leq v+ C$ for some constant $C$, so the result is a counterpart for fully nonlinear equations of the Liouville comparison principles for quasilinear equations, see e.g. \cite{KurtaKawohl} (where $C=0$).

Note also that the constant $0$ is a solution of \eqref{P-} and \eqref{P+} because $c^\alpha\equiv 0$,  so the Corollary contains Theorem \ref{main}  as a special case.
\end{rem}
\begin{rem}
Sufficient conditions for the existence of the Lyapunov functions can be found in \cite{BC} for the Euclidean case, in \cite{BG_lio1} for the Heisenberg group, and 
 in the next sections for other sets of vector fields $\X$. 

Note that Corollary \ref{lcpl}  seems to be new even for linear equations
\[
-\mathrm{Tr}(A(x)(D^2_\X u)^*)+b(x)\cdot D_\X u = 0 
 \text{ in }\R^d
\]
as soon as the operator is degenerate elliptic (and $A$ is the square of a Lipschitz matrix, $b$ Lipschitz).
\end{rem}


\section{
General Carnot groups and sub-Laplacians}
\label{sec_gen}
\subsection{Definitions and preliminaries.}

\smallskip
{A stratified group (or Carnot group) 
 \cite[Definition 2.2.3]{BLU} is a connected and simply connected Lie group whose Lie algebra $\mathcal{G}$ admits a stratification $\mathcal{G}=V_1\oplus V_2\oplus...\oplus V_r$ with $[V_i,V_{i-1}]=V_i$ for $2\leq i\leq r$ and $[V_1,V_r]=0$, where $[V,W]:=\mathrm{span}\{[v,w]: v\in V, w\in W\}$. 
Here, $r$ is called the step of the group, and we set $m=d_1=\mathrm{dim}(V_1)$, $d_i=\mathrm{dim}(V_i)$, $2\leq i\leq r$. A stratified group can be identified with a so-called homogeneous Carnot group up to an isomorphism (see \cite[Section 2.2.3]{BLU}
)}. A  {\it homogeneous Carnot group} $\mathbb{G}$ { (cf. \cite[Definition 1.4.1]{BLU})} can be {in turn} identified with $\R^d=\R^{d_1}\times...\times\R^{d_r}$ ($d=d_1+...+d_r$) endowed with a group law $\circ$ 
 if for  any $\lambda>0$ the dilation $\delta_\lambda:\R^d\to\R^d$ of the form $\delta_\lambda(x)=(\lambda x^{(1)},...,{\lambda^{r}}x^{(r)})$ 
is an automorphism of the group, where $x=(x^{(1)},....,x^{(r)})$, { $x^{(i)}\in\R^{d_i}$}. 
Then the smooth vector fields on $\R^d$ $\X=\{X_1,...,X_m\}$ generate the homogeneous Carnot group $(\R^d,\circ,\delta_\lambda)$ 
if they are left-invariant on $\G$ and such that $X_j(0)=\partial_{x_j}|_0$ for $j=1,...,d_1$ span $\R^d$ at every point $x\in\R^d$. 
We say that $\G$ has step $r$ and $m=d_1$ generators.

In this case, the second order differential operator $\Delta_{\mathbb{G}}=\sum_{i=1}^mX_i^2$ sum of squares of vector fields is called sub-Laplacian on $\mathbb{G}$. We also denote with 
$$
Q:=\sum_{i=1}^rid_i{ =\sum_{i=1}^ri \mathrm{dim}(V_i)}
$$ 
the {\it homogeneous dimension} of the group, with $D_{\mathbb{G}}u=D_\X u:= (X_1u,...,X_mu)$ the horizontal gradient, and with $D^2_{\mathbb{G}}u=D^2_\X u$ the horizontal Hessian. 
We now collect some basic properties that will be needed in the sequel and that can be found in \cite[Chapter 5]{BLU}.
\begin{defn}
Let $\Delta_\G$ be a sub-Laplacian on a homogeneous Carnot group $\mathbb{G}$. We call $\Delta_\G$-gauge a 
symmetric norm $\rho$, { homogeneous with respect to $\delta_\lambda$}, smooth out of the origin,  and satisfying
\[
\Delta_\G(\rho^{2-Q})=0\text{ in }\mathbb{G}\backslash\{0\}\ .
\]
Then, a {\em radial function} on $\mathbb{G}$ is a function $u:\mathbb{G}\backslash\{0\}\to\R$ such that $u(x)=f(\rho(x))$ for all $x\in \mathbb{G}$ for a given $f:(0,\infty)\to \R$ and a $\Delta_\G$-gauge $\rho$ on $\mathbb{G}$.
\end{defn}
\begin{defn}
Let $\G$ be a homogeneous Carnot group on $\R^d$ and $\Delta_\G$ be a sub-Laplacian on $\G$. A function $\Gamma:\R^d\backslash\{0\}\to\R$ is a fundamental solution for $\Delta_\G$ if it satisfies
\begin{itemize}
\item $\Gamma\in C^\infty(\R^d\backslash\{0\})$ ,
\item $\Gamma\in L^1_{loc}(\R^d)$ and $\Gamma\to0$ as $|x|\to\infty$ ,
\item 
$\int_{\R^d}\Gamma\Delta_\G \varphi\,dx=-\varphi(0)$ for all $\varphi\in C_0^\infty(\R^d)$.
\end{itemize}
\end{defn}
We have the following important existence result due to G.B. Folland \cite[Theorem 2.1]{Folland} (see also \cite[Theorem 5.3.2]{BLU} and the references therein)
\begin{thm}
\label{Foll}
Let $\Delta_\G$ be a sub-Laplacian on a homogeneous Carnot group with homogeneous dimension $Q>2$. Then, there exists a fundamental solution $\Gamma$ of $\Delta_\G$. Moreover, the solution is unique.
\end{thm}
We now recall the following result connecting $\Delta_\G$-gauges and fundamental solutions of sub-Laplacians. 
\begin{prop}\label{fund}
Let $\Delta_\G$ be a sub-Laplacian on a homogeneous Carnot group $\mathbb{G}$ and $\Gamma$ be the fundamental solution of $\Delta_\G$. Then 
\begin{itemize}
\item[(i)] 
\[
\rho(x):=\begin{cases}
(\Gamma(x))^{\frac{1}{2-Q}}&\text{ if }x\in \mathbb{G}\backslash\{0\}\ ,\\
0&\text{ if }x=0\ ,
\end{cases}
\]
is a $\Delta_\G$-gauge on $\mathbb{G}$;
\item[(ii)] if $\rho$ is a $\Delta_\G$-gauge on $\mathbb{G}$, then there exists a positive constant $\gamma_d$ such that $\Gamma=\gamma_d \rho^{2-Q}$ is the fundamental solution of $\Delta_\G$;

\item[(iii)] if $\rho$ is a 
 norm on $\mathbb{G}$ { homogeneous with respect to $\delta_\lambda$}, smooth out of the origin,  and such that
\[
\Delta_\G(\rho^\alpha)=0\text{ in }\mathbb{G}\backslash\{0\}
\]
for a suitable $\alpha\in\R$, $\alpha\neq0$, then $\alpha=2-Q$ and $\rho$ is a $\Delta_\G$-gauge on $\mathbb{G}$.
\end{itemize}
\end{prop}
\begin{proof}
(i) is proved in \cite[Proposition 5.4.2]{BLU}, (ii) in \cite[Theorem 5.5.6]{BLU} , and  (iii) in \cite[Corollary 9.9.8]{BLU}
\end{proof}
The next result is proved in \cite[Proposition 5.4.3]{BLU} and allows to compute sub-Laplacians on Carnot groups over radial functions with respect to the homogeneous norm.
\begin{prop}\label{subrad}
Let $\Delta_\G$ be a sub-Laplacian on a homogeneous Carnot group $\mathbb{G}$ and $f=f(\rho(x))$ be a smooth radial function on $\mathbb{G}\backslash\{0\}$. Then
\[
\Delta_\G(f(\rho))=|D_\G \rho|^2\left(f''(\rho)+\frac{Q-1}{\rho}f'(\rho)\right) \,.
\]
\end{prop}
\subsection{Liouville property for some concave-convex equations}
\label{sec_cc}

In this section we derive some Liouville-type results in general homogeneous Carnot groups 
for subsolutions of equations of the form
\begin{equation}
\label{genCarnot+}
F((D^2_\mathbb{G}u)^*)+H_i(x,u,D_\G u)=0 \quad\text{ in }\R^d\ ,
\end{equation}
where
\begin{equation}
\label{FgenCarnot+}
\exists \, \lambda>0 \, : \, F(M)\geq -\lambda\mathrm{Tr}(M) \quad\forall \,M\in \Sym_m \,.
\end{equation}
Note that the Pucci's maximal operator $F=\mathcal{M}^+_{\lambda,\Lambda}$ satisfies such condition, and this motivates the title of this section.
The property \eqref{FgenCarnot+} allows to work by comparing $F((D^2_\mathbb{G}u)^*)$ with the corresponding sub-Laplacian, even if 
a fundamental-type solution {of the second order operator in} \eqref{genCarnot+} is not known in general.
\begin{cor}
\label{quasi_Carnot}
{ 
Assume  \eqref{FgenCarnot+} and that the $\Delta_\G$-gauge $\rho$ 
 satisfies}
\begin{equation}\label{condgen}
\sup_{\alpha\in A}\left\{\rho b^\alpha(x)\cdot D_\G\rho-c^\alpha(x)\rho^2\log\rho\right\}\leq -\lambda(Q-2)|D_\G\rho|^2
\end{equation}
for $\rho$ sufficiently large. 
{Then \eqref{LP1} holds for \eqref{genCarnot+} with $w(x)=\log\rho(x)$.}
%
\end{cor}
\begin{proof}
By \eqref{FgenCarnot+} 
$
\mathcal{M}^+_{\lambda,\Lambda}((D^2_\mathbb{G}u)^*)\geq 
-\lambda\Delta_\G u
$
and so $u$ is a subsolution to the equation $-\Delta_\G u+H_i(x,u,D_\G u)=0$.
{ We {then }apply Theorem \ref{main} with the Lyapunov function 
 $w(x)
 =\log\rho(x)$. }
 Note that $\lim_{|x|\to\infty}w
 (x)=\infty$ because 
  $\rho=\Gamma^{1/(2-Q)} \to\infty$ as $|x|\to\infty$ by { Proposition \ref{fund}}. 
By Proposition \ref{subrad} 
 $w$ satisfies
 \[
-\lambda\Delta_\G(w(\rho))=-|D_\G \rho|^2\frac{(Q-2)\lambda}{\rho^2}\ .
\]
Thus, $w$ is a supersolution at all points where
\begin{equation*}
-|D_\G \rho|^2\frac{(Q-2)\lambda}{\rho^2}+\inf_{\alpha\in A}\left\{c^\alpha(x)\log\rho-b^\alpha(x)\cdot\frac{D_\G\rho}{\rho}\right\}\geq0\ .
\end{equation*}
In particular, this inequality holds when $\rho$ is sufficiently large under condition \eqref{condgen}. 
\end{proof}
\begin{rem}
\label{conv-conc}
A symmetric result holds for supersolutions of $F((D^2_\mathbb{G}u)^*)+H_s(x,u,D_\G u)=0$ if $F(M)\leq -\Lambda\mathrm{Tr}(M)$ for some $\Lambda>0$, as it is the case for the Pucci's minimal operator $F=\mathcal{M}^-_{\lambda,\Lambda}$.
\end{rem}

\begin{rem}
\label{conv-conc}
Under condition \eqref{condgen} the Liouville comparison principle Corollary \ref{lcpl} applies 
to  
 \eqref{genCarnot+} with $w=\log\rho$ and to  the equation of Remark \ref{conv-conc}  with $W=-\log\rho$.
\end{rem}

\subsection{Failure of the Liouville property for linear equations.}

Using the { existence  of a $\Delta_\G$-gauge of Proposition \ref{fund} 
 we can easily 
 show the failure of the Liouville property for sub- and supersolutions to sub-Laplace equations on any homogeneous Carnot group. 
 %
\begin{prop}
\label{nonex}
Let $\X$ 
 be a system of smooth vector fields in $\R^d$ generating a homogeneous Carnot group $\G$ with  homogeneous dimension $Q$, and $\rho$ be a $\Delta_\G$-gauge  on $\mathbb{G}$. Then
  \[
u_1(x)=(1+\rho^2(x))^{1-\frac{Q}{2}}\ ,
\]
is 
a non-constant bounded classical supersolution 
 to $-\Delta_{\G}u=0$ in $\mathbb{G}$, while $-u_1$ is a non-constant bounded subsolution to $-\Delta_{\G}u=0$ in $\mathbb{G}$.  
\end{prop}
\begin{proof}
Clearly $0\leq u_1 \leq 1$ because $Q> 2$, while owing to Proposition \ref{subrad} we find
\begin{multline*}
\Delta_{\G}u_1 
=
|D_\G\rho|^2\left[Q(Q-2)(1+\rho^2)^{-1-\frac{Q}{2}}\rho^2+(2-Q)(1+\rho^2)^{-\frac{Q}{2}}+(1+\rho^2)^{-\frac{Q}{2}}(2-Q)(Q-1)\right]\\
= - |D_\G\rho|^2Q(Q-2)(1+\rho^2)^{-1-\frac{Q}{2}}
\leq 0
\end{multline*}
\end{proof}
}
\begin{rem}
When $\mathbb{G}$ is the classical $d$-dimensional Euclidean group, {$d\geq 3$}, the sub-Laplace operator reduces to the classical Laplacian $\Delta:=\sum_{i=1}^d\partial_{x_i}^2$, the homogeneous dimension is $Q=d$ and $\rho(x)=|x|$. 
This case is classical, and counterexamples to the one-side Liouville property { can be found in {\cite{BG_lio1} and the references therein}.
 }
\end{rem}

\begin{rem}
{ One- or two-side} Liouville properties for solutions to sub-Laplace equations on Carnot groups have been investigated in \cite{ICDCsub} by means of mean-value formulas and in \cite{BLU} via Harnack inequalities. Proposition \ref{nonex} shows that, instead, 
 the property may fail for mere sub- or supersolutions. 
\end{rem}

We  consider now a family of smooth vector fields fulfilling the following assumptions
\begin{itemize}
\item[(H1)] $X_1,...,X_m$ are linearly independent and satisfy the H\"ormander rank condition at $0\in\R^d$.
\item[(H2)] $X_1,...,X_m$ are homogeneous of degree 1 with respect to a family of anisotropic dilations $\delta_\lambda$ on $\R^d$ of the form
\[
\delta_\lambda(x)=(\lambda^{\sigma_1}x_1,...,\lambda^{\sigma_d}x_d)
\]
for $\sigma_i\in\N$, $i=1,...,d$ and $1\leq \sigma_1\leq...\leq\sigma_d$.
\end{itemize}
We can extend Proposition \ref{nonex} 
to sub-Laplacians built on vector fields fulfilling the above assumptions (H1) and (H2) via a global lifting argument that dates back to 
 Folland and recently reconsidered in \cite{BB}. 
{ A similar argument has been recently used in \cite[Proposition 5.6]{BiagiLanconelli} to prove non-existence of nontrivial solutions for sub-Laplace equations, and the next result shows that, instead, non-constant sub- and supersolutions do exist.}
\begin{cor}
Let $\X=\{X_1,...,X_m\}$ be a system of smooth vector fields in $\R^d$ satisfying (H1)-(H2). Then, there exist non-constant 
{ classical subsolutions bounded above and supersolutions bounded below } to $-\sum_{j=1}^mX_j^2u=0$ in $\mathbb{R}^d$.
\end{cor}
\begin{proof}
The 
 lifting argument in 
\cite[Theorem 3.2]{BB} 
shows that there exists a homogeneous Carnot group { $\mathbb{G}\simeq \R^N$}, $N>d$, and a system of Lie generators { $Z_1,...,Z_m$ of $\G$}, which are the lifted vector fields corresponding to $X_1,...,X_m$. { Then, we set $\Delta_\G=\sum_{j=1}^mZ_j^2$ and observe that $\Delta_\G (u\circ\pi)=(\Delta_\X u)\circ \pi$, where $\pi:\R^N\to\R^d$ is the canonical projection of $\R^N$ onto the first $d$ variables. Then, if $u$ solves $-\Delta_\X u\leq0$ in $\R^d$, the function $v=u\circ \pi$ solves $-\Delta_\G v\leq0$ in $\G\simeq\R^N$. 
Therefore, the failure of the Liouville property readily follows from Proposition \ref{nonex} on $\R^N$, which gives the existence of a non-constant bounded smooth subsolution to $-\Delta_\X u=0$ on $\R^d$ when the fields satisfy the hypotheses (H1)-(H2). }
\end{proof}

\subsection{{ The distance from a constant at infinity}}
{ 
When the Liouville property does not hold one would to like to evaluate how far is a sub- or supersolution from 
{the constant of the Liouville property}. In the Euclidean case this was done by V. Kurta for a class of quasilinear equations in \cite{Kurta} via integral methods (see also \cite{Veron} for similar test function methods applied to semilinear equations on the Heisenberg group) 
{and more recently by L. D'Ambrosio and E. Mitidieri in \cite{DM19}}. Let us first note that for the classical Laplacian in dimension $d\geq3$ the supersolution 
 of Proposition \ref{nonex} is $u_1(x)=(1+|x|^2)^{\frac{2-d}{2}}$, which satisfies
\[
\liminf_{r\to\infty}\left[\sup_{r\leq|x|\leq 2r}(u_1(x)-0)\right]r^{d-2}=C .
\]
Theorem 2 in \cite{Kurta} implies that, for any 
 $\nu\in(0,1)$,  a superharmonic function $u\geq 0$ either is constant in $\R^d$ or it satisfies 
\begin{equation}\label{condk}
\liminf_{r\to\infty}\left[\sup_{r\leq|x|\leq 2r}(u(x)-0)\right]r^{\frac{d-2}{1-\nu}}=+\infty .
\end{equation}
This gives the ``sharp distance at infinity" of nonconstant supersolutions larger than $c$ from the constant $c$.

Next we give the corresponding result in a homogeneous Carnot group, where we expect the Euclidean dimension $d$ to be replaced by the homogeneous dimension $Q$. In fact, the supersolution  of Proposition \ref{nonex} satisfies
\[
\liminf_{r\to\infty}\left[\sup_{r\leq\rho(x)\leq 2r}(u(x)-0)\right]r^{Q-2}=C .
\]
\begin{thm}\label{sharpdist}
Let $u\geq c$ be a non-constant { weak} solution to $-\Delta_\G u\geq 0$ in $\G\simeq \R^d$ { such that $u\in L^\infty_{\mathrm{loc}}(\R^d)$}. 
Then, for any $\nu\in(0,1)$, 
\begin{equation}\label{condksub}
\liminf_{r\to\infty}\left[\sup_{r\leq\rho(x)\leq 2r}(u(x)-c)\right]r^{\frac{Q-2}{1-\nu}}=+\infty
\end{equation}
\end{thm}
\begin{rem}
By weak supersolution we mean a measurable function $u:\G\simeq \R^d\to\R$ defined on $\G$ which is locally integrable in $\R^d$, $|D_\G u|\in L^2_{\mathrm{loc}}(\R^d)$ and 
\[
\int_{\R^d}\sum_{i=1}^m X_i \varphi X_i u\,dx\geq0
\] 
for all nonnegative test functions $\varphi\in W^{1,2}(\G)$ {with compact support}, where $W^{1,q}(\G)$, $q>1$, is the standard horizontal Sobolev space. 
\end{rem}
}

\begin{proof}
The proof relies on proving {the following stronger property: if $u$ is a non-constant supersolution to $-\Delta_\G u=0$ in $\R^d$, then it satisfies the condition}
\begin{equation}\label{condk2}
\liminf_{r\to\infty}r^{-2}\int_{r\leq\rho(x)\leq 2r}(u(x)-c)^{1-\nu}\,dx=+\infty
\end{equation}
for any {fixed} $\nu\in(0,1)$. Then, the latter would imply \eqref{condksub} if {we further have $u\in L^\infty_{\mathrm{loc}}(\R^d)$}.

To show the condition \eqref{condk2} one adapts to Carnot groups the proof in the Euclidean case of 
 \cite[Theorem 2]{Kurta}. We summarize here below the crucial steps, and refer to the earlier version of this manuscript\footnote{arXiv: 2109.11331v1, September 23, 2021} and to \cite{Kurta} for a thorough discussion. The idea is to use a variational argument, taking the test function
\[
\varphi(x)=(u-c+\epsilon)^{-\nu}\zeta^2(x)\ ,\nu\in (0,1)
\] 
where $\zeta:\R^d\to [0,1]$ is smooth and identically 1 on the Koranyi ball $\overline{B}_\G(0,r)$ and 0 outside $B_\G(0,2r)$, to be specialized later. Rearranging the weak formulation of the problem and using Cauchy-Schwarz, H\"older and Young inequalities one deduces that
\begin{equation}\label{final}
\liminf_{r\to\infty}\int_{A_r} 
|D_\G \zeta|^2(u-c)^{1-\nu}=+\infty\ .
\end{equation}
This step concludes the proof by choosing $\zeta(x)=\psi(\rho/2r)$, where $\rho$ is a homogeneous norm of the structure, $\psi:[0,+\infty]\to[0,1]$ is a smooth function which equals 1 on $[0,1/2]$, 0 on $[1,\infty)$ and satisfies the inequality
$ |D_\G \zeta|\leq Cr^{-1}
$
(since $|D_\G\rho|\leq \tilde C$ in view of the fact that $D_\G\rho$ is homogeneous of degree 0, cf \cite[Remark 5.5.2]{BLU}) for arbitrary $r>0$. 
\end{proof}
\begin{rem}
Theorem \ref{sharpdist} agrees with the growth rate of superharmonic functions found in \cite[Theorem 1.1]{Biri}, at least in the context of the Heisenberg group. 
{Moreover, a more general decay rate of positive supersolutions has been recently studied in \cite{DM19} for operators in divergence form with regular coefficients. This can be obtained combining Corollary 23 and Remark 12 in \cite{DM19}.}
\end{rem}

\begin{rem}
Using a slight modification of the above argument and following \cite{Kurta}, one can prove that any nonconstant {and locally bounded weak} solution to {$-\Delta_{\G,p} u:=-\mathrm{div}_\G(|D_\G u|^{p-2}D_\G u)\geq 0$ in $\G$, where $\mathrm{div}_\G$ stands for the horizontal divergence along the vector fields,} with $u\geq c$ and $1<p<Q$, satisfies
\begin{equation}\label{condksubp}
\liminf_{r\to\infty}\left[\sup_{r\leq\rho(x)\leq 2r}(u(x)-c)\right]r^{\frac{Q-p}{p-1-\nu}}=+\infty
\end{equation}
with any fixed $\nu\in(0,p-1)$. {This agrees with the Liouville property found in \cite[Corollary 4.3]{Da} when $p\geq Q$.}
\end{rem}

\section{Fully nonlinear PDEs on Carnot groups of step 2}
\label{sec_fully}
\subsection{
Definitions and preliminaries}
\label{Carnot}
In this section, we introduce step-2 Carnot groups following \cite[Chapter 3]{BLU}. Set $\R^d=\R^m\times\R^n$. Given $n$ skew-symmetric matrices $B^{(1)},....,B^{(n)}$ with dimension $m\times m$ and real entries, one defines the group operation
\[
(x,  t)\bullet (x',  t'
)=(x+x',  t+t'+2\langle Bx, x'\rangle) \,,
\]
where, $\langle Bx, x'\rangle$ stands for the $n$-tuple
\[
(\langle B^{(1)}x,x'\rangle,...,\langle B^{(n)}x,x'\rangle)
\]
and $\langle \cdot,\cdot\rangle$ denotes the inner product in $\R^m$. Then, $(\R^d,\bullet)$ is a Lie group, while the dilation
\[
\delta_\lambda:\R^d\to\R^d\ ,\delta_\lambda((x,t))=(\lambda x,\lambda^2 t)
\]
is an automorphism of $(\R^d,\bullet)$ for any positive $\lambda$. Then $\G=(\R^d,\bullet,\delta_\lambda)$ is a homogeneous Lie group. When $B^{(s)}=(b^{(s)}_{jl})$, $j,l=1,...,m$, and for ${q}=(x_1,...,x_m,t_1,...,t_n)$, a basis of the Lie algebra of $\G$ is given by the $m+n$ left invariant vector fields
\[
X_j({q})=\partial_{x_j}+2\sum_{s=1}^n\sum_{l=1}^mb^{(s)}_{jl}x_l\partial_{t_s}\ ,j=1,...,m\ ,
\]
\[
Y_j({q})=\partial_{t_s}\ ,s=1,...,n
\]
As pointed out in \cite{BLU}, we have $n\leq \frac{m(m-1)}{2}$, that allows to relate the dimension of the horizontal and vertical layers. 
\begin{ex} 
The Heisenberg group is a step 2 Carnot group with $m=2d$ and $n=1$ on $\R^{2d+1}$ by taking
\[
B^{(1)}=\begin{pmatrix}
0 &I_{d}\\
-I_d & 0
\end{pmatrix} \, .
\]
\end{ex}
\begin{ex} 
 H-type groups. They are Carnot groups of step 2 
such that the matrices $B^{(s)}$ satisfy the further conditions
\begin{itemize}
\item $B^{(s)}$ are $m\times m$ skew symmetric and orthogonal. 
\item $B^{(k)}B^{(s)}+B^{(s)}B^{(k)}=0$ for $k,s=1,...,n$ with $k\neq s$.
\end{itemize}
In this case the homogeneous dimension is defined as $Q=m+2n$. {Prototype} examples of matrices fulfilling the above conditions, {cf. \cite[p.687]{BLU}}, are
\begin{equation}\label{matH}
B^{(1)}=\begin{pmatrix}
0 & -1 & 0 & 0\\
1& 0& 0& 0\\
0 & 0& 0 &-1\\
0 & 0 & 1 & 0
\end{pmatrix}\ ;
B^{(2)}=\begin{pmatrix}
0 & 0 & 1 & 0\\
0& 0& 0& -1\\
-1 & 0& 0 &0\\
0 & 1 & 0 & 0
\end{pmatrix}\ ;
B^{(3)}=\begin{pmatrix}
0 & 0 & 0 & 1\\
0& 0& 1& 0\\
0 & -1& 0 &0\\
-1 & 0 & 0 & 0
\end{pmatrix}
\end{equation}
which generate a H-type group on $\R^7=\R^4\times\R^3$,  {and we will focus on them in the next sections}. {Other explicit examples can be found in \cite[Remark 3.6.6]{BLU}}.
{Also the Heisenberg group  is 
a H-type group} with $m=2d$ and $n=1$ on $\R^{2d+1}$, by taking $B^{(1)}$ as in the previous example.

For H-type groups a gauge is
\begin{equation}
\label{rhoHtype}
\rho_{\mathbb{H}}(x):=\left[(x_1^2+...+x_m^2)^2+ t_{1}^2+...+t_{n}^2\right]^{\frac{1}{4}}
\end{equation}
and Kaplan \cite{Kaplan} (see also \cite{BLU}) proved that 
 $\Gamma_\He:=\rho_\He^{2-Q}$, $Q=m+2n$, is the fundamental solution to the corresponding sub-Laplace equation. 
 (Here we made the change of coordinates for the vertical layer 
 $t'=t/4$ to normalize a constant in the original definition of $\rho_{\mathbb{H}}$ in \cite{Kaplan}).
\end{ex}
\begin{ex} 
\label{free}
Free Carnot groups $\mathbb{F}_{r}=\mathbb{F}_{r,2}\sim\R^r\times \R^{\frac{r(r-1)}{2}}$. 
 They are Carnot groups of step 2 
where the matrices $B^{(s)}$ have the entry $-1$ in position $(j,l)$, $+1$ in position $(l,j)$, and $0$ otherwise.  In $\R^d$, let us denote the coordinates of the first layer by $x_k$, $1\leq k\leq r$ and those of the second layer by $t_{kj}$, $1\leq j<k\leq r$. Let $\partial_k$ and $\partial_{kj}$ the standard basis vectors in this coordinate system. Free Carnot groups of step-2 and $r$ generators are spanned by the vector fields
\begin{equation*}
X_k:=\partial_k+2\left(\sum_{j>k}x_j\partial_{jk}-\sum_{j<k}x_j\partial_{kj}\right)\text{ if }1\leq k\leq r\ .
\end{equation*}
Then, denote by
\begin{equation*}
X_{kj}:=\partial_{kj}\text{ if }1\leq j<k\leq r\ .
\end{equation*}
the fields of the vertical layer. The Carnot structure of $\mathbb{F}_r$ is given by
\begin{equation*}
V_1=\text{Span}\{X_k:1\leq k\leq r\}\text{ and }V_2=\text{Span}\{X_{kj}:1\leq j<k\leq r\}
\end{equation*}
The commutation relations for $1\leq j<k\leq r$ and $1\leq i\leq r$ are given by
\begin{equation*}
[X_k,X_j]=4X_{kj}\text{ and }[X_i,X_{kj}]=0
\end{equation*}
Denote the coordinates by
\begin{equation*}
(x,t)=(x_H,x_V)=(x_1,...,x_{r},t_{2,1},...,t_{r,r-2}, t_{r,r-1})\in\R^r\times\R^{\frac{r(r-1)}{2}}
\end{equation*}
and consider the homogeneous norm
\begin{equation}\label{rf}
\rho_\mathbb{F}(x)=\left[|x_H|^4+
|x_V|^2\right]^{\frac{1}{4}}\ .
\end{equation}
It is immediate to recognize that free Carnot group of step-2 coincides with the Heisenberg group only in one dimension, i.e. $\He^1\equiv(\R^3,\cdot)$. Indeed for $r=2d$ generators, we have a free Carnot group of step 2 if and only if the following equality holds
$2d+1=2d+
{2d(2d-1)}/{2}\ ,$
which is fulfilled only when $d=1$.
\end{ex}

\subsection{Fundamental solutions to Pucci's extremal equations on a class of H-type groups}\label{sec_fund}
In this section and the next we consider the specific case of matrices \eqref{matH}. Then $m=4, n=3$, $Q=10$, $\rho=\rho_{\mathbb{H}}$ is given by \eqref{rhoHtype}, and we denote the coordinates
\begin{equation*}
(x,t)=(x_H,x_V)=(x_1,...,x_4,t_1,...,t_3)  . 
\end{equation*}
It is immediate to check that the algebra is spanned by the vector fields
\begin{equation}\label{fieldsH}
\begin{aligned}
& X_1=\partial_{x_{1}}-2x_{2}\partial_{t_1}+2x_3\partial_{t_2}+2x_4\partial_{t_3} , \quad 
X_2=\partial_{x_{2}}+2x_{1}\partial_{t_1}-2x_4\partial_{t_2}+2x_3\partial_{t_3} ,\\
& X_3=\partial_{x_3}-2x_{4}\partial_{t_1}-2x_1\partial_{t_2}-2x_2\partial_{t_3} , \quad 
 X_4=\partial_{x_{4}}+2x_{3}\partial_{t_1}+2x_2\partial_{t_2}-2x_1\partial_{t_3}\ .
\end{aligned}
\end{equation}
We now give two 
 simple algebraic lemmas for later use. 
 The first is taken from 
  \cite{LeoniJMPA}.
\begin{lemma}\label{eiggen}
Let $a,b,c,s\in\R$ and $v,w\in \R^m$ be unitary vectors and define
\[
A=sI_m+av\otimes v+bw\otimes w+c(v\otimes w+w\otimes v)\ .
\]
Then, the eigenvalues of $A$ are
\begin{itemize}
\item $s$ with multiplicity at least $m-2$ and eigenspace $\mathrm{Span}\{v,w\}^{\bot}$,  
i.e.,  the orthogonal complement of $ \mathrm{Span}\{v,w\}$;
\item $s+\frac{a+b+2cv\cdot w\pm \sqrt{(a+b+2cv\cdot w)^2+4(1-(v\cdot w)^2)(c^2-ab)}}{2}$ which are simple, if different from $s$.
\end{itemize}
In particular, when either $(v\cdot w)^2=1$ or $c^2=ab$, the eigenvalues are $s$ with multiplicity $m-1$ and $s+a+b+2cv\cdot w$, which is simple.
\end{lemma}

\begin{lemma}\label{hesH}
We have $X_i\rho=\frac{\mu_i}{\rho^3}$ where $\mu_i$, $i=1,...,4$, are defined via
\[
\mu_1=-x_2t_1+x_3t_2+x_4t_3\ ,\mu_2:=x_1t_1-x_4t_2+x_3t_3\ ,\mu_3:=-x_4t_1-x_1t_2-x_2t_3\ ,\mu_4:=x_3t_1+x_2t_2-x_1t_3\ .
\]
Moreover, $|D_{\He}\rho|^2=|x_H|^2/\rho^2$,
\[
(D^2_{\He}\rho)^*=3\frac{|D_\He\rho|^2}{\rho}I_4-\frac{3}{\rho}D_\He\rho\otimes D_\He\rho
\]
and
\[
(D^2_{\He}f(\rho))^*=3\frac{f'(\rho)}{\rho}|D_\He\rho|^2I_4+\left(f''(\rho)-3\frac{f'(\rho)}{\rho}\right)D_\He\rho\otimes D_\He\rho
\]
In particular, the eigenvalues of $(D^2_{\He}f(\rho))^*$ are $3\frac{f'(\rho)}{\rho}|D_\He\rho|^2$ with multiplicity 3 and $f''(\rho){|D_\He\rho|^2}$ which is a simple eigenvalue. Finally
\[
\Delta_\He f(\rho)=|D_\He\rho|^2\left(f''(\rho)+9\frac{f'(\rho)}{\rho}\right)
\]
where $9=Q-1$, $Q=10$ being the homogeneous dimension.
\end{lemma}
\begin{proof}
We compute
\[
X_1\rho=\frac{x_1|x_H|^2}{\rho^3}+\frac{1}{\rho^3}(-x_2t_1+x_3t_2+x_4t_3)=:\frac{\mu_1}{\rho^3}\ ;\ X_2\rho=\frac{x_2|x_H|^2}{\rho^3}+\frac{1}{\rho^3}(x_1t_1-x_4t_2+x_3t_3)=:\frac{\mu_2}{\rho^3}
\]
\[
X_3\rho=\frac{x_3|x_H|^2}{\rho^3}+\frac{1}{\rho^3}(-x_4t_1-x_1t_2-x_2t_3)=:\frac{\mu_3}{\rho^3}\ ;\ X_4\rho=\frac{x_4|x_H|^2}{\rho^3}+\frac{1}{\rho^3}(x_3t_1+x_2t_2-x_1t_3)=:\frac{\mu_4}{\rho^3}\ .
\]
Then, it is immediate to check that
\begin{multline*}
|D_\He\rho|^2=\sum_{i=1}^4(X_i\rho)^2=\frac{|x_H|^6}{\rho^6}+\frac{|x_H|^2(t_1^2+t_2^2+t_3^2)}{\rho^6}\\
+\frac{2}{\rho^6}(-x_2x_3t_1t_2+x_3x_4t_2t_3-x_2x_4t_1t_3-x_1x_4t_1t_2-x_3x_4t_2t_3+x_1x_3t_1t_3+x_4x_1t_1t_2\\
+x_4x_2t_1t_3+x_1x_2t_2t_3+x_3x_2t_1t_2-x_1x_2t_2t_3-x_1x_3t_1t_3)\\
+2\frac{|x_H|^2}{\rho^6}(-x_1x_2t_1+x_3x_1t_2+x_1x_4t_3+x_1x_2t_1-x_4x_2t_2+x_3x_2t_3-x_4x_3t_1-x_1x_3t_2\\
-x_2x_3t_3+x_3x_4t_1+x_3x_4t_1+x_2x_4t_2-x_1x_4t_3)\\
=\frac{|x_H|^6}{\rho^6}+\frac{|x_H|^2(t_1^2+t_2^2+t_3^2)}{\rho^6}=\frac{|x_H|^2}{\rho^6}\rho^4=\frac{|x_H|^2}{\rho^2}
\end{multline*}
We then compute
\[
X_i^2\rho=\frac{|D_\He\rho|^2}{\rho}+2\frac{|x_H|^2}{\rho^3}-\frac{3}{\rho}X_i\rho X_i\rho\ ; X_2X_1\rho=-\frac{t_1}{\rho^3}-\frac{3}{\rho}X_2\rho X_1\rho\ ; X_1X_2\rho=\frac{t_1}{\rho^3}-\frac{3}{\rho}X_1\rho X_2\rho
\]
\[
X_2X_3\rho=-\frac{t_3}{\rho^3}-\frac{3}{\rho}X_2\rho X_3\rho\ ;X_3X_2\rho=\frac{t_3}{\rho^3}-\frac{3}{\rho}X_3\rho X_2\rho\ ; X_2X_4\rho=\frac{t_2}{\rho^3}-\frac{3}{\rho}X_2\rho X_4\rho
\]
\[
X_4X_2\rho=-\frac{t_2}{\rho^3}-\frac{3}{\rho}X_4\rho X_2\rho\ ; X_4X_3\rho=-\frac{t_1}{\rho^3}-\frac{3}{\rho}X_4\rho X_3\rho\ ; X_3X_4\rho=\frac{t_1}{\rho^3}-\frac{3}{\rho}X_3\rho X_4\rho
\]
Then, using the fact that $|D_\He\rho|^2\rho^2=|x_H|^2$, the symmetrized horizontal Hessian takes the form
\[
(D^2_{\He}\rho)^*=3\frac{|D_\He\rho|^2}{\rho}I_4-\frac{3}{\rho}D_\He\rho\otimes D_\He\rho\ .
\]
We then conclude that for a radial function $f=f(\rho)$ we have
\[
(D^2_{\He}f(\rho))^*=3\frac{f'(\rho)}{\rho}|D_\He\rho|^2I_4+\left(f''(\rho)-3\frac{f'(\rho)}{\rho}\right){|D_\He \rho|^2\frac{D_\He\rho}{|D_\He \rho|}\otimes \frac{D_\He\rho}{|D_\He \rho|}.}
\]
We then use {Lemma \ref{eiggen} applied with $m=4$, $s=3\frac{f'(\rho)}{\rho}|D_\He\rho|^2$, $a=\left(f''(\rho)-3\frac{f'(\rho)}{\rho}\right)|D_\He \rho|^2$, $b=c=0$} to conclude that the eigenvalues are $3\frac{f'(\rho)}{\rho}|D_\He\rho|^2$ with multiplicity 3 and {$f''(\rho)|D_\He \rho|^2$,} which is a simple eigenvalue.
\end{proof}

We are now able to compute explicitly the
degenerate Pucci's extremal operators of radial functions. 
\begin{lemma}\label{pucciHtype}
For every $f=f(\rho)$ concave and increasing we have
\[
\mathcal{M}^+_{\lambda,\Lambda}((D^2_{\He}f(\rho))^*)=-|D_\He\rho|^2\left(\lambda(Q-1)\frac{f'(\rho)}{\rho}+\Lambda f''(\rho)\right)
\]
while for $f=f(\rho)$ convex and decreasing we have
\[
\mathcal{M}^{{+}}_{\lambda,\Lambda}((D^2_{\He}f(\rho))^*)=-|D_\He\rho|^2\left(\lambda f''(\rho)+\Lambda(Q-1)\frac{f'(\rho)}{\rho}\right)\ ,
\]
where $Q=10$ is the homogeneous dimension of the H-type group {associated with the matrices \eqref{matH}}.
\end{lemma}
\begin{proof}
This is a simple consequence of Lemma \ref{hesH} and the formulas for the Pucci's extremal operators.
\end{proof}
We are now ready to construct 
 {classical solutions}, that we call ``fundamental solutions", of
\[
\mathcal{M}^{{\pm}}_{\lambda,\Lambda}((D^2_{\He}u)^*)=0\text{ in }\R^7\backslash\{0\}\ ,
\]
This extends to a different subelliptic structure some results of \cite{CLeoni} and the more general ones in \cite{ArmstrongCPAM} for the uniformly elliptic case, 
as well as  those in \cite{CTchou} for the 
 Heisenberg group. 
  In particular, by taking $\lambda=\Lambda$, our fundamental solutions of the corresponding sub-Laplacian agrees with those found by Kaplan \cite{Kaplan}. These classical solutions depend on the two parameters
\[
\alpha :=\frac{\lambda}{\Lambda}(Q-1)+1=9\frac{\lambda}{\Lambda}+1\geq 1
\]
and 
\[
\beta :=\frac{\Lambda}{\lambda}(Q-1)+1=9\frac{\Lambda}{\lambda}+1\geq 10=Q\ .
\]
Standard calculations similar to those carried out in \cite{CLeoni}, {\cite[Lemma 3.3]{CTchou}} lead to the following
\begin{lemma}
\label{fundsol}
The radial functions
$ 
\Phi_1(x)=\varphi_1(\rho)\text{ and }\Phi_2(x)=\varphi_2(\rho)
 $ 
with
\[
\varphi_1(\rho):=\begin{cases}
C_1\rho^{2-\alpha}+C_2&\text{ for }\alpha<2\\
C_1\log\rho+C_2&\text{ for }\alpha=2\\
-C_1\rho^{2-\alpha}+C_2&\text{ for }\alpha>2\\
\end{cases}
\]
and
\[
\varphi_2(\rho):=C_1\rho^{2-\beta}+C_2
\]
are classical (and hence viscosity) solutions to
\[
\mathcal{M}^+_{\lambda,\Lambda}((D^2_{\He}u)^*)=0\text{ in }\R^7\backslash\{0\}\ .
\]
{In particular, $\varphi_1$ is concave and increasing, while $\varphi_2$ is convex and decreasing with respect to $\rho$.}
Similarly, $\Psi_1=-\Phi_2$ and $\Psi_2=-\Phi_1$ are ``fundamental solutions" of
\[
\mathcal{M}^-_{\lambda,\Lambda}((D^2_{\He}u)^*)=0\text{ in }\R^7\backslash\{0\}\,.
\]
\end{lemma}

\subsection{Liouville properties on H-type groups}
 We consider the specific case of matrices \eqref{matH} as in the previous section, {and refer to the next Remark \ref{gen} for the general case}.
We begin with the homogeneous Pucci's equation { and improve Corollary \ref{quasi_Carnot} in this special case. The next result  is known in the Euclidean and Heisenberg cases with different proofs, see \cite[Theorem 3.2]{CLeoni}, \cite[Theorem 4.3]{AScpde} and \cite[Theorem 5.2]{CTchou}, and it is new for  H-type groups. We provide a {new} unified proof with the same strategy as in \cite{BG_lio1}.}
\begin{thm}\label{liohad}
Let $\X$ be either the Euclidean or  the Heisenberg vector fields, or the generators of the H-type groups in \eqref{fieldsH}, and $Q$ be the corresponding homogeneous dimension. Let $u$ be a continuous viscosity solution to  $\mathcal{M}^+_{\lambda,\Lambda}((D^2_{\X}u)^*)\leq 0$ (resp., $\mathcal{M}^-_{\lambda,\Lambda}((D^2_{\X}u)^*) \geq 0$) in $\R^d$ bounded from above (resp., below). 
If $Q\leq \frac{\Lambda}{\lambda}+1$, then $u$ is constant. 
\end{thm}
\begin{proof}
We outline the main steps 
in the case of the maximal operator, the other being similar. Let $\rho$ be the homogeneous norm associated to $\X$. Assume temporarily the existence of a function $w\in C^2(\R^d\backslash\{0\})$ such that  $\mathcal{M}^+_{\lambda,\Lambda}((D^2_{\X}w)^*) \geq 0$ for $\rho(x)>R$ and $\lim_{|x|\to\infty}w=+\infty$. 
For $\xi>0$ set $v_\xi(x)=u(x)-\xi w(x)$ for some $\bar R>R>0$. Then $v_\xi$ is continuous on $\{x\in\R^d:\rho(x)>\bar R\}$ and solves in the viscosity sense $\mathcal{M}^-_{\lambda,\Lambda}((D^2_{\X}v_\xi)^*)\leq 0$ for $\rho(x)>\bar R$ by the properties of Pucci's extremal operators. Indeed, 
if both $u,w$ are $C^2$, we have
\[
\mathcal{M}^-_{\lambda,\Lambda}((D^2_{\X}v_\xi)^*)\leq \mathcal{M}^+_{\lambda,\Lambda}((D^2_{\X}u)^*) +\mathcal{M}^-_{\lambda,\Lambda}(-\xi (D^2_{\X}w)^*)=\mathcal{M}^+_{\lambda,\Lambda}((D^2_{\X}u)^*) -\xi \mathcal{M}^+_{\lambda,\Lambda}((D^2_{\X}w)^*)\leq0 ,
\]
and the inequality is true in the viscosity sense by an argument in \cite[Theorem 2.1]{BG_lio1}.
Therefore,  arguing again as in \cite
{BG_lio1} 
by the weak comparison 
 principle on the annuli, one observes that $u$ attains its maximum at some point on $\partial B(0,\bar R)$. Then, one applies the strong maximum principle \cite{BG} to conclude that $u$ is constant.

It remains to verify the existence of a Lyapunov function $w$ for the equation. We take $w(x)=\log\rho(x)$ ($\rho(x) = |x|$ in the Euclidean case and $\rho = \rho_{\He}$ in the other cases) and claim that
\[
\mathcal{M}^+_{\lambda,\Lambda}((D^2_{\X}w)^*)= ({\Lambda}-\lambda(Q-1))
\frac{|D_\X\rho(x)|^2}{\rho(x)^2} .
\]
This is well known for the Euclidean and Heisenberg case, and it follows from Lemma \ref{pucciHtype} for the case of H-type groups. Then
\[
\mathcal{M}^+_{\lambda,\Lambda}((D^2_{\X}w)^*)
\geq0   \quad \text{for $\rho(x)>0$ }\iff Q\leq \frac{\Lambda}{\lambda}+1 \,.
\]
\end{proof}
\begin{rem} The condition on $Q$ in Theorem \ref{liohad} is sharp, as it can be seen immediately by constructing explicit counterexamples as in \cite{BG_lio1}.
Note that it is equivalent to the condition $\alpha
\leq 2$ and therefore to the unboudedness at infinity of the fundamental solution built in  Lemma \ref{fundsol}.
We stress again 
 that the one-side Liouville property for sub- and supersolutions to sub-Laplacians ($\lambda=\Lambda$) fails in any Carnot group of step-2.
 
  The proof of Theorem \ref{liohad} 
 in \cite{CTchou} used a Hadamard-type result in the Heisenberg group. Such a qualitative property can be obtained also for H-type groups, and Theorem \ref{liohad} can also be deduced  from it.
\end{rem}
We now consider equations \eqref{P-} and \eqref{P+} structured over the H-type vector fields, namely,
\begin{equation}\label{1H}
\mathcal{M}^-_{\lambda,\Lambda}((D^2_{\He}u)^*)+H_i(x,u,D_{\He}u)=0 \quad \text{ in }\R^7\ ,
\end{equation}
\begin{equation}\label{2H}
\mathcal{M}^+_{\lambda,\Lambda}((D^2_{\He}u)^*)+H_s(x,u,D_{\He}u)=0 \quad \text{ in }\R^7\ ,
\end{equation}
where $H_i, H_s$ are defined by \eqref{Hi} and \eqref{Hs}. In the rest of this section $\rho = \rho_{\He}$ defined by \eqref{rhoHtype}.
\begin{thm}\label{Cor1H}
Let $\mathcal{X}=\{X_1,....,X_{4}
\}$ be the system of vector fields generating the H-type group $\He\simeq\R^7$ associated to the matrices in \eqref{matH}. Assume that \eqref{b} and \eqref{c} hold and
\begin{equation}\label{condH}
\sup_{\alpha\in A}\{b^\alpha(x)\cdot \frac{\mu}{|x_H|^2}-c^\alpha(x)\frac{\rho^4}{|x_H|^2}\log\rho\}\leq \lambda-9\Lambda
\end{equation}
for $\rho$ sufficiently large and $|x_H|\neq0$, where $Q=10$ is the homogeneous dimension of $\mathbb{H}\simeq\R^7$, $b^\alpha(x)$ takes values in $\R^{4}$, and $\mu$ is defined in Lemma \ref{hesH}. 
\begin{itemize}
\item[(A)] If either $c^\alpha(x)\equiv0$ or $u\geq0$, then \eqref{LP1} holds for \eqref{1H} with $w(x)=\log\rho(x)$.
\item[(B)] If either $c^\alpha(x)\equiv0$ or $v\leq0$, then \eqref{LP2} holds for \eqref{2H}  with $W(x)=-\log\rho(x)$.
\end{itemize}
As a consequence, \eqref{LP1} and \eqref{LP2} hold for $G$ uniformly subelliptic under the assumptions \eqref{G>H} or \eqref{G<H} and $H_i,H_s$ as in \eqref{Hi} or \eqref{Hs}.
\end{thm}
\begin{proof}
We only have to check the existence of a Lyapunov function and apply Theorem \ref{main}.
Take
 $w(x)=\log\rho(x)$. Note that $\lim_{|x|\to\infty}w
 (x)=\infty$ because 
 $\rho\to\infty$ as $|x|\to\infty$.
By Lemma \ref{pucciHtype} we can compute the Pucci minimal operator 
for  $\rho(x)>0$ 
and get
\begin{equation*}
\mathcal{M}^-_{\lambda,\Lambda}((D^2_{\He}w)^*)=\{-\Lambda(Q-1)+\lambda\}\frac{|D_\He\rho|^2}{\rho^2}\ .
\end{equation*}
Thus, $w$ is a supersolution at all points where
\begin{equation*}
\{-\Lambda(Q-1)+\lambda\}\frac{|x_H|^2}{\rho^4}+\inf_{\alpha\in A}\{c^\alpha(x)\log\rho-b^\alpha(x)\cdot \frac{\mu}{\rho^4}\}\geq0\ ,
\end{equation*}
{because $|D_\He\rho|^2=|x_H|^2/\rho^2$ and $D_{\He}w= \mu/\rho^4$ by Lemma \ref{hesH}.}  
In particular, this inequality holds when $\rho$ is sufficiently large under condition \eqref{condH} by recalling that $Q=10$. Similarly, one can check that \eqref{condH} implies that the function $W(\rho)=-\log\rho$ is a subsolution to \eqref{2H} for $\rho(x)>0$. 
Therefore Theorem \ref{main} gives the conclusion.
\end{proof}
\begin{rem}
It is easy to give simpler sufficient conditions on the coefficients so that condition \eqref{condH} holds. For instance,
\[
\limsup_{|x|\to\infty}\sup _{\alpha\in A} b^\alpha(x)\cdot \frac{\mu}{|x_H|^2} < \lambda-9\Lambda \, ,
\]
since $c^\alpha\geq 0$, or, uniformly in $\alpha$,
\[
c^\alpha(x) \geq c_o>0 \,, \quad |b^\alpha(x)| = o\left(
{\rho^2 \log \rho}/|x_H|
\right) \, ,
\]
because $|\mu|=\rho^2 |x_H|$  (the condition on $b^\alpha$ is satisfied, for instance, if they are uniformly bounded), guarantee the validity of \eqref{condH}. 
\end{rem}
\begin{rem}\label{gen}
The Liouville properties obtained above for the matrices \eqref{matH} can be extended to 
general H-type groups under similar conditions by means of the expression of the symmetrized horizontal Hessian of a rescaled  homogeneous norm 
 given in \cite[p. 468]{Tralli}. Moreover, the same computations would lead to new results even for linear equations driven by $-\mathrm{Tr}(A(x)(D^2_\He u)^*)$, where $A$ is uniformly positive definite with eigenvalues between two given constants $\lambda<\Lambda$ via \cite[Theorem 2.1]{BG_lio1}. In both cases, to prove \eqref{LP1} one can take the Lyapunov functions $w_1(x)=\log\rho$ or $w_2(x)=\rho^2$ {(cf. \cite[Remark 4.6]{BG_lio1})} and get sufficient conditions involving the dimension of the underlying structure.
\end{rem}
\begin{rem}
The previous sufficient condition is comparable with the one obtained in \cite{BC} for the Euclidean vector fields and in \cite{BG_lio1} for the Heisenberg vector fields, since the constant on the right-hand side in \eqref{condH} can be written as
$
\lambda-9\Lambda=\lambda-\Lambda(Q-1)
$.
It can be improved for subsolutions (resp. supersolutions) when $\mathcal{M}^-_{\lambda,\Lambda}$ (resp. $\mathcal{M}^+_{\lambda,\Lambda}$) is replaced with 
 $\mathcal{M}^+_{\lambda,\Lambda}$ (resp. $\mathcal{M}^-_{\lambda,\Lambda}$), as we did in Theorem \ref{liohad}. {This is consistent with non-existence results already found for different equations, see, e.g., \cite[Remark 2.6]{CirantGoffi} and the references therein}.
For instance,  consider a subsolution $u$ to the fully nonlinear equation perturbed by linear terms 
\begin{equation}\label{3H}
\mathcal{M}^+_{\lambda,\Lambda}((D^2_{\He}u)^*)+b(x)\cdot D_{\He} u+c(x)u=0 \quad \text{ in }\R^7\ ,
\end{equation}
where $\mathcal{X}$ are either the Heisenberg vector fields or 
those generating the H-type group $\He\simeq\R^7$ associated to the matrices in \eqref{matH}. Assume 
\eqref{b}, \eqref{c},  and
\begin{equation}\label{condHimp}
\limsup_{|x|\to\infty}\left\{b(x)\cdot \frac{\mu}{|x_H|^2}-c(x)\frac{\rho^4}{|x_H|^2}\log\rho(x) \right\}\leq \Lambda-\lambda(Q-1)
\end{equation}
for $\rho$ sufficiently large, 
 where $Q$ is the homogeneous dimension of the structure, and $\mu$ is defined in Lemma \ref{hesH} for $H$-type groups or in \cite[eq. (37)]{BG_lio1} for the Heisenberg group. If $\limsup_{|x|\to\infty}
 \frac{u(x)}{\log\rho} \leq 0$ and either $c(x)\equiv0$ or $u\geq0$, then $u$ is constant. The proof is the same as that of Theorem \ref{liohad}.
\end{rem}

We can deduce from Theorem \ref{Cor1H} and Corollary \ref{lcpl} the following Liouville comparison principle.
\begin{cor}
[Liouville comparison principle] 
\label{comparisonH}
Let $u, v$ be a sub- and a supersolution 
to \eqref{1H} (resp., \eqref{2H}) such that $\limsup_{|x|\to\infty}\frac{u(x) - v(x)}{\log\rho(x)}\leq 0$
 and assume $c^\alpha\equiv0$ and \eqref{condH} holds. Then $u\equiv v$ up to constants in $\R^d$.
\end{cor}
\begin{proof}
By \eqref{condH} and the proof of Thm.  \ref{Cor1H}, $w=\log\rho$ is a Lyapunov function for  \eqref{1H} and $W=-w$ is a negative Lyapunov function, $u-v$ is a subsolution and $v-u$ a supersolution of \eqref{1H}. Then one can conclude by the Liouville properties of Thm.  \ref{Cor1H}.
\end{proof}

\subsection{Equations with horizontal Hessian and Euclidean gradient}
\label{subsec_mix} 
In 
 this section, {following \cite{BG_lio1}}, we  consider equations that involve the H-Hessian $(D_{
\He}^2u)^*$ and the 
 Euclidean gradient $Du$, namely of the form
\begin{equation}\label{non-h-grad-2}
G(x,u, Du, (D_{
\He}^2u)^*)=0\text{ in }\R^d=\R^{m+n}\ ,
\end{equation}
with $G : \R^{m+n}\times\R\times\R^{m+n}\times\mathcal S_{m} \to\R$.
We impose $G$ uniformly subelliptic and its first order part  bounded from below by a concave Hamiltonian $H_i$, or from above by a convex one $H_s$, which are allowed to depend on 
 the full Euclidean gradient. More precisely,
\begin{equation}\label{G>Hp}
{ G(x,r,p,0)}\geq H_i(x,r,p) ,  \quad \forall \, x, p\in \R^d , r\in\R ,
\end{equation}
for a concave Hamiltonian of the form \eqref{Hi}, but with  $b^\alpha : \R^d \to \R^d$ a vector field in $\R^d$,  or 
\begin{equation}\label{G<Hp}
{ G(x,r,p,0)}\leq H_s(x,r,p)     ,  \quad \forall \, x, p\in \R^d , r\in\R ,
\end{equation}
for a convex Hamiltonian of the form \eqref{Hs} with  $b^\alpha : \R^d \to \R^d$. On $b^\alpha$ and $c^\alpha$ we make the same assumptions \eqref{b}, \eqref{c}. 

The next result provides an explicit sufficient condition on H-type groups associated to the matrices \eqref{matH}. 
It states that the velocity fields $b^\al$ in the drift part of the Hamiltonian terms $H_i$ and $H_s$ point toward the origin for $|x|$ sufficiently large, and it is expressed in terms of the homogeneous norm $\rho$ of the H-type group defined by \eqref{rhoHtype}.%
\begin{cor}\label{cor_euc1}
Assume that the operator $G$ { satisfies \eqref{unifsubell}, 
where
$\X=\{X_1,...,X_{4}\}$}  are the H-type vector fields in $\R^7$ given by \eqref{fieldsH},  
and \eqref{b} and  \eqref{c} hold.  
Suppose there exist $\g_1,\dots, \g_{m+n}\in\R$ with $\min_i\g_i=\g_o>0$ and such that 
\begin{equation}\label{OUtype}
\sup_\al b^\al(x)\cdot D\rho(x) \leq -\sum_{i=1}^{m+n}\g_ix_i\partial_i\rho + o\left(\frac1{\rho^3} \right)\qquad\text{ as }\rho\to \infty.
\end{equation}
\noindent {\upshape(A)} Assume 
 \eqref{G>Hp}.
 If either $c^\alpha(x)\equiv0$ or $u\geq0$, then \eqref{LP1} holds for \eqref{G>Hp} with $w(x)= \log\rho(x)$.

\noindent {\upshape(B)} Assume 
 \eqref{G<Hp}.
 If either $c^\alpha(x)\equiv0$ or $v\leq0$, then \eqref{LP2} holds  for \eqref{G<Hp} with $W(x)=-\log\rho(x)$.
\end{cor}
\begin{proof}
We only need to check that $w=\log\rho$ is a supersolution to
\[
\mathcal{M}^-_{\lambda,\Lambda}((D^2_{\He}u)^*)+H_i(x,u,Du)=0\ .
\]
Let $C_1:= 9\Lambda-\lambda > 0$. As in the proof of Theorem \ref{Cor1H}, $w$ is a supersolution at all points where
\begin{equation}\label{tobecheck}
-C_1\frac{|x_H|^2}{\rho^4}+\inf_{\alpha\in A}\left\{c^\alpha(x)\log\rho-b^\alpha(x)\cdot\frac{D\rho}{\rho}\right\}\geq0\ .
\end{equation}
Since $D\rho = (2|x_H|^2 x_H, x_V)/(2\rho^3)\in\R^{m+n}$, we deduce from \eqref{OUtype}  that the left-hand side of the above inequality is greater than 
\begin{multline*}
-C_1\frac{|x_H|^2}{\rho^4}+\frac 1{2\rho^4}\left( 2\sum_{i=1}^{m}\g_ix_i^2|x_H|^2+ \sum_{j=1}^{n}\g_{j}x_{j}^2 + o(1)\right)\geq
\\
\frac 1{\rho^4}\left( |x_H|^2 (\g_o|x_H|^2 - C_1) +\frac{\g_o}2 |x_{V}|^2+ o(1)\right) \geq 0 \,,
\end{multline*}
for $\rho$ large enough, by taking either $|x_H|^2 > C_1/\g_o$, or $|x_H|^2 \leq C_1/\g_o$ and $|x_{V}|^2>2C_1^2/\g_o^2$. 
\end{proof}
As a byproduct, we conclude with a result based on a positivity condition of the zero-th order coefficients $c^\al$ at infinity similar to {\cite[Example 4.9 and Corollary 4.10]{BG_lio1}. We skip the proof since it is the same of \cite[Corollary 4.10]{BG_lio1}}.
\begin{cor}\label{cor_euc2}
In the assumptions of Corollary \ref{cor_euc1} replace \eqref{OUtype} with 
\begin{equation}\label{c_a}
\liminf_{|x|\to\infty} \inf_{\alpha\in A}c^\alpha(x)
\log\rho(x)>0 
 \,,
\end{equation}
and either
\begin{equation}\label{sign}
\limsup_{|x|\to\infty} \sup_{\alpha\in A}b^\alpha(x)\cdot D\rho(x)\leq 0 \,,
\end{equation}
or 
\begin{equation}\label{order}
 \sup_{\alpha\in A}|b^\alpha(x)|  = o(\rho) \qquad\text{ as }\rho\to \infty.
\end{equation}
Then the conclusions of Corollary \ref{cor_euc1} hold.
\end{cor}
\begin{rem}
Corollary \ref{cor_euc1} generalizes to fully nonlinear equations over H-type vector fields the Liouville properties proved 
in the companion paper \cite{BG_lio1} for the Heisenberg group. 

The condition \eqref{c_a} in Corollary \ref{cor_euc2} obviously holds if $c^\alpha(x)\geq c_o>0$ for $|x|$ large enough, and in such case the condition \eqref{order} can be weakened to $\sup_{\alpha\in A}|b^\alpha(x)|  = o(\rho\log \rho)$.
\end{rem}


\subsection{Liouville properties on free step-2 Carnot groups}\label{sec_free} 

Here we consider the Carnot groups $\mathbb{F}_r$ with $r$ parameters introduced in Example \ref{free} of Section \ref{Carnot} with $\rho_\mathbb{F}$ defined by \eqref{rf}. 
\begin{lemma}
\label{hessianfree}
The gauge $\rho_\mathbb{F}$ in $\mathbb{F}_r$ satisfies
\[
(D^2_{\mathbb{F}_r}\rho_\mathbb{F})^*=\frac{3|x_H|^2}{\rho_\mathbb{F}^4}I_r-\frac{3}{\rho_\mathbb{F}^2}D_{\mathbb{F}_r}\rho_\mathbb{F}\otimes D_{\mathbb{F}_r}\rho_\mathbb{F}\ .
\]
In particular, for a radial function $f=f(\rho_\mathbb{F})$ we have
\[
(D^2_{\mathbb{F}_r}f(\rho_\mathbb{F}))^*=f'(\rho_\mathbb{F})\frac{3|x_H|^2}{\rho_\mathbb{F}^4}I_r+\left(f^{''}(\rho)-\frac{3f'(\rho_\mathbb{F})}{\rho_\mathbb{F}^2}\right) D_{\mathbb{F}_r}
\rho_\mathbb{F}\otimes D_{\mathbb{F}_r}\rho_\mathbb{F}
\]
and
\[
\Delta_{\mathbb{F}_r}f(\rho_\mathbb{F})=f'(\rho_\mathbb{F})\frac{3|x_H|^2}{\rho_\mathbb{F}^4}r+\left(f^{''}(\rho_\mathbb{F})-\frac{3f'(\rho_\mathbb{F})}{\rho_\mathbb{F}^2}\right)|D_{\mathbb{F}_r}\rho_\mathbb{F}|^2 .
\]
\end{lemma}
\begin{proof}
We drop the subscript $\mathbb{F}$ in the homogeneous norm $\rho$ for ease of notation. We have
\begin{equation*}
X_k\rho=\frac{1}{\rho^3}\left[x_k\sum_{j=1}^rx_j^2+\left(\sum_{j>k}x_jt_{jk}-\sum_{j<k}x_jt_{kj}\right)\right] .
\end{equation*}
Thus, we obtain
\begin{equation*}
|D_{\mathbb{F}_r}\rho|^2=\frac{|x_H|^6}{\rho^6}+\frac{1}{\rho^6}\sum_k\left(\sum_{j>k}x_jt_{jk}-\sum_{j<k}x_jt_{kj}\right)^2 .
\end{equation*}
Moreover, we can compute
\[
X_k(X_k\rho)=\frac{1}{\rho^3}\left[\sum_{j=1}^r x_j^2+2x_k^2+2\sum_{j=1,j\neq k}^rx_j^2\right]-\frac{3}{\rho}X_k\rho X_k\rho=\frac{3|x_H|^2}{\rho^3}I_r-\frac{3}{\rho}X_k\rho X_k\rho
\]
and
\[
X_i(X_k\rho)=\frac{1}{\rho^3}\left[2x_ix_k-t_{ki}-2x_ix_k\right]-\frac{3}{\rho}X_i\rho X_k\rho\text{ for }i<k;
\]
\[
X_i(X_k\rho)=\frac{1}{\rho^3}\left[2x_ix_k+t_{ik}-2x_ix_k\right]-\frac{3}{\rho}X_i\rho X_k\rho\text{ for }i>k\ .
\]
Therefore, the horizontal hessian $D^2_{\mathbb{F}_r}\rho\in\R^{r\times r}$ is given by
\begin{equation*}
D^2_{\mathbb{F}_r}\rho=\frac{1}{\rho^3}\left[T+3|x_H|^2I_r\right]-\frac{3}{\rho}D_{\mathbb{F}_r}\rho\otimes D_{\mathbb{F}_r}\rho\ ,
\end{equation*}
where $x_H=(x_1,...,x_r)$ and $T$ is the skew-symmetric matrix
\begin{equation*}
T:=\begin{pmatrix}0 &-t_{21}&\hdots&-t_{r1}\\t_{21} & 0 & \hdots&\hdots\\ \vdots & \vdots & \vdots&\vdots \\t_{r1}& \hdots &\hdots&0\end{pmatrix}
\end{equation*}
We then conclude
\begin{equation*}
(D^2_{\mathbb{F}_r}\rho)^*=\frac{3|x_H|^2}{\rho^4}I_r-\frac{3}{\rho^2}D_{\mathbb{F}_r}\rho\otimes D_{\mathbb{F}_r}\rho\ .
\end{equation*}

\end{proof}

 We now consider fully nonlinear problems of the form
\begin{equation}\label{1G}
\mathcal{M}^-_{\lambda,\Lambda}((D^2_{\mathbb{F}_r}u)^*)+H_i(x,u,D_{\mathbb{F}_r}u)=0 \quad \text{ in }\R^d\ ,
\end{equation}
\begin{equation}\label{2G}
\mathcal{M}^+_{\lambda,\Lambda}((D^2_{\mathbb{F}_r}u)^*)+H_s(x,u,D_{\mathbb{F}_r}u)=0 \quad \text{ in }\R^d\ .
\end{equation}

\begin{thm}
\label{corfree}
Assume \eqref{b}-\eqref{c} and
\begin{equation}\label{condcor1free}
\sup_{\alpha\in A}\left\{b^\alpha(x)\cdot\frac{D_{\mathbb{F}_r}\rho}{\rho}-c^\alpha(x)\log\rho\right\}\leq \frac{4\lambda}{\rho^2}|D_{\mathbb{F}_r}\rho|^2-\Lambda\frac{3r}{\rho^4}|x_H|^2
\end{equation}
for $\rho$ large enough. 
\begin{itemize}
\item[(a)] Assume that either $c^\alpha(x)\equiv0$ or $u\geq0$, then \eqref{LP1} holds for \eqref{1G} with $w(x)=\log\rho(x)$.
\item[(b)] Assume that either $c^\alpha(x)\equiv0$ or $v\leq0$, then \eqref{LP2} holds for \eqref{2G} with $W(x)=-\log\rho(x)$.
\end{itemize}
\end{thm}
\begin{proof} The proof follows the same lines at that of Theorem \ref{Cor1H} with the Lyapunov function $w:=\log\rho$, $\rho=\rho_{\mathbb{F}}$.
%
By Lemma \ref{hessianfree} and the chain rule we get 
\begin{equation}\label{horhessfree}
(D^2_{\mathbb{F}_r}\log(\rho))^*=\frac{3|x_H|^2}{\rho^4}I_r-\frac{4}{\rho^2}D_{\mathbb{F}_r}\rho\otimes D_{\mathbb{F}_r}\rho =: N+M .
\end{equation}
By {Lemma \ref{eiggen}} the eigenvalues of $M$ are $-\frac{4}{\rho^2}|D_{\mathbb{F}_r}\rho|^2$, which is simple, and $0$ with multiplicity $r-1$, while the eigenvalue of $N$ is $3|x_H|^2/\rho^4$ with multiplicity $r$.
By the sub-additivity of the Pucci's minimal operator (cf \cite{CC}) we get
\begin{multline*}
\mathcal{M}^-_{\lambda,\Lambda}(D^2_{\mathbb{F}_r}w)+H_i(x,w,D_{\mathbb{F}_r}w) \geq \mathcal{M}^-_{\lambda,\Lambda}\left(-\frac{4}{\rho^2}D_{\mathbb{F}_r}\rho\otimes D_{\mathbb{F}_r}\rho\right)+\mathcal{M}^-_{\lambda,\Lambda}\left(\frac{1}{\rho^4}(3|x_H|^2I_r)\right)\\
+\inf_{\alpha\in A}\left\{c^\alpha(x)\log\rho-b^\alpha(x)\cdot\frac{D_{\mathbb{F}_r}\rho}{\rho}\right\}=\frac{4\lambda}{\rho^2}|D_{\mathbb{F}_r}\rho|^2-\Lambda\frac{3r}{\rho^4}|x_H|^2\\
+\inf_{\alpha\in A}\left\{c^\alpha(x)\log\rho-b^\alpha(x)\cdot\frac{D_{\mathbb{F}_r}\rho}{\rho}\right\}\ .
\end{multline*}
Hence $w$ is a supersolution of \eqref{1G} at all $x$ where \eqref{condcor1free} holds. 
\end{proof}
%
\begin{rem} 
\label{F-H}
This result can be compared with  the one obtained in \cite{BG_lio1} for 
 the Heisenberg group $\mathbb{H}^1$, which is a free step two Carnot group 
 with $r=2$. 
We set $\bar \eta:= \rho^3D_{\mathbb{F}_r}\rho$ and recall that in this case $|D_{\mathbb{F}_2}\rho|^2=\frac{|x_H|^2}{\rho^2}$ by \cite[Lemma 3.1]{CTchou}. Then \eqref{condcor1free} reads
\begin{equation*}
\sup_{\alpha\in A}\{b^\alpha(x)\cdot \bar \eta-c^\alpha(x)\rho^4\log\rho\}\leq \left(4\lambda-6\Lambda\right)|x_H|^2 .
\end{equation*}
On the other hand, the sufficient condition in \cite[Theorem 4.2]{BG_lio1} is 
\begin{equation*}
\sup_{\alpha\in A}\{b^\alpha(x)\cdot \bar \eta-c^\alpha(x)\rho^4\log\rho\}\leq (\lambda-3\Lambda)|x_H|^2\ .
\end{equation*}
The two conditions coincide only if $\lambda=\Lambda$, i.e., when the operators reduce to the sub-Laplacian. Then for the nonlinear case
$
\lambda<\Lambda$ Theorem \ref{corfree} is not sharp, as expected, 
because in the proof we had to use the sub- and superadditivity inequalities of the extremal operators.
\end{rem}

We now address the extremal operators $\mathcal{P}^\pm_\lambda$, introduced by Pucci in \cite{Pucci66}, computed over horizontal Hessians. These operators can be written in terms of  the ordered eigenvalues of the matrix $M\in\Sym_m$ as follows 
\begin{equation*}
\label{representation+}
\mathcal{P}^{+}_{\lambda}(M)=\sup_{A\in\mathcal{B}_\lambda}-\mathrm{Tr}(AM)=-\lambda\sum_{k=2}^{m
}e_k-[1-(m
-1)\lambda]e_1=-\lambda\mathrm{Tr}(M)-(1-m
\lambda)e_1 
\,,
\end{equation*}
\begin{equation*}
\label{representation-}
\mathcal{P}^{-}_{\lambda}(M)=\inf_{A\in\mathcal{B}_\lambda}-\mathrm{Tr}(AM)=-\lambda\sum_{k=1}^{m-1}e_k-[1-(m-1)\lambda]e_m=-\lambda\mathrm{Tr}(M)-(1-m\lambda)e_m\, .
\end{equation*}
where
\[
\mathcal{B}_\lambda=\{A\in\mathrm{Sym}_m: A\geq \lambda I_d\ ,\mathrm{Tr}(A)=1\} .
\]
Note that they can be computed by knowing  only the trace and an extremal eigenvalue of the Hessian, different from $\mathcal{M}^\pm$ that require the sign of all the eigenvalues.
We have the following {sharp} Liouville property for the equations
\begin{equation}\label{pucci1bis}
\mathcal{P}^-_{\lambda}((D^2_{\mathbb{F}_r}u)^*)+H_i(x,u,D_{\mathbb{F}_r}u)=0\quad \text{ in }\R^{d}\ ,
\end{equation}
\begin{equation}
\label{pucci2bis}
\mathcal{P}^+_{\lambda}((D^2_{\mathbb{F}_r}u)^*)+H_s(x,u,D_{\mathbb{F}_r}u)=0\quad \text{ in }\R^{d}\ .
\end{equation}
\begin{cor}
\label{Cor2P}
For $H_i, H_s$ defined by \eqref{Hi}, \eqref{Hs} with $b^\alpha : \R^d\to\R^r$, 
assume
 \eqref{b}-\eqref{c} and
\begin{equation}\label{condcor1freepucci} 
\sup_{\alpha\in A}\{b^\alpha(x)\cdot \bar \eta-c^\alpha(x)\rho^4\log\rho\}\leq -3|x_H|^2+4\lambda\rho^2|D_{\mathbb{F}_r}\rho|^2\ ,
\end{equation}
for $\rho$ sufficiently large, 
where $\bar \eta :=\rho^3 D_{\mathbb{F}_r}\rho$.
\begin{itemize}
\item[(c)] Assume that either $c^\alpha(x)\equiv0$ or $u\geq0$, then \eqref{LP1} holds for \eqref{pucci1bis} with $w(x)=\log\rho(x)$.
\item[(d)] Assume that either $c^\alpha(x)\equiv0$ or $v\leq0$, then \eqref{LP2} holds for \eqref{pucci2bis} with $W(x)=-\log\rho(x)$.
\end{itemize}
\end{cor}
\begin{proof}
The proof 
follows the same line as that of Theorem \ref{Cor1H} with $\rho=\rho_{\mathbb{F}}$. 
By {Lemma \ref{eiggen}} the eigenvalues of the horizontal Hessian in \eqref{horhessfree} are $\frac{3|x_H|^2}{\rho^4}$ with multiplicity $r-1$ and $\frac{3|x_H|^2}{\rho^4}-\frac{4}{\rho^2}|D_{\mathbb{F}_r}\rho|^2$ which is simple. Then
\[
\mathcal{P}^-_{\lambda}((D^2_{\G_r}\log\rho)^*)=\frac{4\lambda}{\rho^2}|D_{\mathbb{F}_r}\rho|^2-\frac{3|x_H|^2}{\rho^4} .
\]
\end{proof}

\begin{rem}
As in Remark \ref{F-H} we compare with  the result obtained in 
 \cite{BG_lio1} for 
 the Heisenberg group $\mathbb{H}^1$. Now
condition \eqref{condcor1freepucci} is the same as the one in 
 \cite[Corollary 4.10]{BG_lio1}, so Corollary \ref{Cor2P} is sharp.
\end{rem}
\begin{rem}
The Liouville comparison principle Corollary \ref{comparisonH} continues to hold for equations \eqref{1G} and \eqref{2G} assuming \eqref{condcor1free} instead of \eqref{condH}, and for equations \eqref{pucci1bis} and \eqref{pucci2bis} assuming \eqref{condcor1freepucci}.
\end{rem}
\begin{rem}
Results for equations involving the horizontal Hessian $D^2_{\mathbb{F}_r}u$ and the Euclidean gradient $Du$ can be found as in Section \ref{subsec_mix}.
\end{rem}

\section{Fully nonlinear equations on Grushin geometries}
\label{sec_gru}
\subsection{Preliminaries on Grushin structures}

The prototype sub-Riemannian geometry of Grushin-type is the one generated by the vector fields 
\begin{equation}
\label{grushin}
X=\partial_{x} \, \quad Y=x\partial_{y}
\end{equation}
for $p=(x,y)\in\R^2$, 
that leads to the so-called Grushin operator defined as
\[
\Delta_\X:=\partial_{x}^2+x^2\partial_{y}^2\ ,\quad (x,y)\in\R^2\ .
\]
Unlike the examples presented in the previous section in the realm of Carnot groups, here the vector fields are not left-invariant with respect to any group action on $\R^2$. However, one can easily check that $X$ and $Y$ satisfy the H\"ormander condition, since at the origin we have $\mathrm{Span}\{X,Y\}=\mathrm{Span}\{X\}\neq\R^2$, but $[X,Y]=\partial_y$ and $\mathrm{Span}\{X,Y,[X,Y]\}=\R^2$ at any point $(x,y)\in\R^2$. In this case, we consider the following homogeneous norm
\begin{equation}
\label{homnorm}
\rho(x,y)=(x^4+4y^2)^{\frac14}\ ,
\end{equation}
which is 1-homogeneous with respect to the dilations $\delta_\lambda(x,y)=(\lambda x,\lambda^2 y)$, $\lambda>0$. 

 A multidimensional counterpart can be defined as follows. Let $\gamma$ be a positive real number and $(x,y)\in\R^d=\R^n_x\times\R^k_y$ with $n,k\geq1$. Consider the vector fields
\begin{equation}
\label{grushin_d}
X_i=\partial_{x_i}\ ,Y_j=|x|^\gamma\partial_{y_j}\ ,i=1,...,n\, , \,j=1,...,k
\end{equation}
and denote by $D_\X =(D_x,|x|^\gamma D_y)$ the horizontal gradient and $\Delta_\X=\Delta_x+|x|^{2\gamma}\Delta_y$ the Grushin sub-Laplacian and, finally, let 
$$
Q:=n+(1+\gamma)k
$$ be the corresponding homogeneous dimension. For such structures we consider the homogeneous norm
\begin{equation}
\label{homnorm_d}
\rho(x,y)=\left(|x|^{2(1+\gamma)}+(1+\gamma)^2|y|^2\right)^{\frac{1}{2+2\gamma}}\ .
\end{equation}
When $\gamma=0$ this generalizes the classical Laplacian, while for $\gamma>0$ the ellipticity of the operator 
 degenerates at $\{0\}\times\R^k\subseteq \R^d$. When $\gamma=2k$, $k\in\N$, $\Delta_\X$ is a sum of squares of $C^\infty$ vector fields fulfilling the H\"ormander's rank condition. 

We start with the following algebraic results on the plane.
\begin{lemma}\label{hessianradial}
Let $\X$ and $\rho$ be defined by \eqref{grushin} and \eqref{homnorm}. Then
$D_\X \rho=(X\rho,Y\rho)
=\frac{1}{\rho^3}\left(x^3,2xy\right)$, 
$
|D_\X\rho|^2=\frac{x^2}{\rho^2}$, and 
\[
(D^2_{\mathcal{X}}\rho)^*=\frac{1}{\rho^3}\begin{pmatrix}3x^2& y\\y & 2x^2\end{pmatrix}-\frac{3}{\rho}D_{\mathcal{X}}\rho\otimes D_{\mathcal{X}}\rho \ .
\]
In particular, for a radial function $f=f(\rho)$ we have
\[
(D^2_{\mathcal{X}}f(\rho))^*=\frac{f'(\rho)}{\rho^3}\begin{pmatrix}3x^2& y\\y & 2x^2\end{pmatrix}+\left(f''(\rho)-\frac{3f'(\rho)}{\rho}\right)D_{\mathcal{X}}\rho\otimes D_{\mathcal{X}}\rho\ .
\]
\end{lemma}
\begin{proof}
The computation of $D_\X\rho$ is easy. 
Then the entries of the intrinsic Hessian $D^2_\X\rho$ are given by
\begin{equation*}
X(X\rho)=\frac{3x^2}{\rho^3}-\frac{3}{\rho}X\rho X\rho\ ;\ X(Y\rho)=\frac{2y}{\rho^3}-\frac{3}{\rho}X\rho Y\rho\ ;\ Y(X\rho)=-\frac{3}{\rho}Y\rho X\rho\ ;\ Y(Y\rho)=\frac{2x^2}{\rho^3}-\frac{3}{\rho}Y\rho Y\rho\ .
\end{equation*}
Therefore, the matrix $D^2_{\mathcal{X}}\rho$ can be written as
\begin{equation*}
D^2_{\mathcal{X}}\rho=\frac{1}{\rho^3}\begin{pmatrix}3x^2& 2y\\0 & 2x^2\end{pmatrix}-\frac{3}{\rho}D_{\mathcal{X}}\rho\otimes D_{\mathcal{X}}\rho\ .
\end{equation*}
and the lemma is proved by exploiting the chain rule $(D^2_{\mathcal{X}}f(\rho))^*=f'(\rho)(D^2_{\mathcal{X}}\rho)^*+f''(\rho)D_\X \rho\otimes D_\X \rho$.
\end{proof}
We now prove an analogous result for the multidimensional generalized Grushin vector fields. 
\begin{lemma}
\label{hesgru} 
Let $\X$ and $\rho$ be defined by \eqref{grushin_d} and \eqref{homnorm_d}. Then
\[
D_\X \rho=\frac{\left(|x|^{2\gamma}x_i,(1+\gamma)y_j|x|^{\gamma}\right)}{\rho^{2\gamma+1}} \,,\quad
|D_\X\rho|^2=\frac{|x|^{2\gamma}}{\rho^{2\gamma}} \,,
\]
\begin{multline*}
(D^2_\X\rho)^*=\frac{|x|^{2\gamma}}{\rho^{2\gamma+1}}I_d+2\gamma\frac{|x|^{2\gamma}}{\rho^{2\gamma+1}}\begin{pmatrix} \frac{x}{|x|}\otimes \frac{x}{|x|} & 0_{\R^{k \times k}}\\ 0_{\R^{k\times k}} & 0_{\R^{k\times k}}\end{pmatrix}-\frac{(2\gamma+1)}{\rho}D_\X\rho\otimes D_\X\rho\\+\gamma\frac{|x|^{2\gamma}}{\rho^{2\gamma+1}}\begin{pmatrix} 0_{\R^n} & 0_{\R^{k\times k}}\\ 0_{\R^{k\times k}} & I_k\end{pmatrix}
+\frac{\gamma(1+\gamma)}{2}\frac{|x|^{\gamma-2}}{\rho^{2\gamma+1}}\begin{pmatrix}0_{\R^{n\times n}} & x\otimes y\\ y\otimes x & 0_{\R^{k\times k}}\end{pmatrix}\ .
\end{multline*}
For a radial function for $f=f(\rho)$
\begin{multline*}
(D^2_{\mathcal{X}}f(\rho))^*=f'(\rho)\frac{|x|^{2\gamma}}{\rho^{2\gamma+1}}I_d+2\gamma f'(\rho)\frac{|x|^{2\gamma}}{\rho^{2\gamma+1}}\begin{pmatrix} \frac{x}{|x|}\otimes \frac{x}{|x|} & 0_{\R^{k \times k}}\\ 0_{\R^{k\times k}} & 0_{\R^{k\times k}}\end{pmatrix}+f'(\rho)\gamma\frac{|x|^{2\gamma}}{\rho^{2\gamma+1}}\begin{pmatrix} 0_{\R^n} & 0_{\R^{k\times k}}\\ 0_{\R^{k\times k}} & I_k\end{pmatrix}\\
+f'(\rho)\frac{\gamma(1+\gamma)}{2}\frac{|x|^{\gamma-2}}{\rho^{2\gamma+1}}\begin{pmatrix}0_{\R^{n\times n}} & x\otimes y\\ y\otimes x & 0_{\R^{k\times k}}\end{pmatrix}
+\left(f''(\rho)-f'(\rho)\frac{(2\gamma+1)}{\rho}\right)D_\X\rho\otimes D_\X\rho
\end{multline*}
and
 \[
\Delta_\X f(\rho)=\frac{|x|^{2\gamma}}{\rho^{2\gamma}}\left(f''(\rho)+f'(\rho)\frac{Q-1}{\rho}\right)
\]
\end{lemma}
\begin{proof}
We compute
\[
X_i\rho=\frac{|x|^{2\gamma}x_i}{\rho^{2\gamma+1}} \, , \quad
Y_j\rho=\frac{(1+\gamma)y_j|x|^{\gamma}}{\rho^{2\gamma+1}}\,,
\]
from which it readily follows that $|D_\X\rho|^2=\frac{|x|^{2\gamma}}{\rho^{2\gamma}}$ using the definition of the gauge norm. Then
\[
X_jX_i\rho=2\gamma\frac{|x|^{2\gamma-2}x_jx_i}{\rho^{2\gamma+1}}+\frac{|x|^{2\gamma}}{\rho^{2\gamma+1}}\delta_{ij}-\frac{2\gamma+1}{\rho}X_j\rho X_i\rho\ ;\ Y_jY_i\rho=(1+\gamma)\frac{|x|^{2\gamma}}{\rho^{2\gamma+1}}\delta_{ij}-\frac{2\gamma+1}{\rho}Y_j\rho Y_i\rho
\]
\[
X_jY_i\rho=\gamma(1+\gamma)\frac{|x|^{\gamma-2}}{\rho^{2\gamma+1}}x_jy_i-\frac{2\gamma+1}{\rho}X_j\rho Y_i\rho\ ;\ Y_jX_i\rho=-\frac{2\gamma+1}{\rho}Y_j\rho X_i\rho
\]
Therefore
\begin{multline*}
(D^2_\X\rho)^*=\frac{|x|^{2\gamma}}{\rho^{2\gamma+1}}I_d+2\gamma\frac{|x|^{2\gamma}}{\rho^{2\gamma+1}}\begin{pmatrix} \frac{x}{|x|}\otimes \frac{x}{|x|} & 0_{\R^{k \times k}}\\ 0_{\R^{k\times k}} & 0_{\R^{k\times k}}\end{pmatrix}-\frac{(2\gamma+1)}{\rho}D_\X\rho\otimes D_\X\rho\\+\gamma\frac{|x|^{2\gamma}}{\rho^{2\gamma+1}}\begin{pmatrix} 0_{\R^{n\times n}} & 0_{\R^{k\times k}}\\ 0_{\R^{k\times k}} & I_k\end{pmatrix}
+\frac{\gamma(1+\gamma)}{2}\frac{|x|^{\gamma-2}}{\rho^{2\gamma+1}}\begin{pmatrix}0_{\R^{n\times n}} & x\otimes y\\ y\otimes x & 0_{\R^{k\times k}}\end{pmatrix}
\end{multline*}
and the result follows by the chain rule.
\end{proof}

 \subsection{Liouville theorems for linear and quasi-linear equations}
We first underline that the Liouville property for the (classical) subsolutions (supersolutions) bounded from above (below) of the mere Grushin sub-Laplace equation does not hold. Indeed, the function
\[
\bar{u}(x)=\begin{cases}
\frac18[15-10\rho^2+3\rho^4]&\text{ if }\rho\leq1\ ,\\
\frac{1}{\rho}&\text{ if }\rho\geq1\ .
\end{cases}
\] 
is a non-constant bounded classical supersolution to $-\Delta_\X u=-\partial_{xx}u-x^2\partial_{yy}u=0$ in $\R^2$. Indeed, when $\rho\leq1$ one checks, setting
\[
u(x)=g(\rho)=\frac18[15-10\rho^2+3\rho^4]
\]
and using Lemma \ref{hessianradial}, that
\[
(D^2_\X g(\rho))^*={\frac12}\left(3-\frac{5}{\rho^2}\right)\begin{pmatrix}3x^2& y\\y & 2x^2\end{pmatrix}+5D_\X\rho\otimes D_\X\rho.
\]
Therefore, exploiting that $\rho\leq1$ one concludes
\[
-\Delta_\X u_5=-\mathrm{Tr}((D^2_\X g(\rho))^*)=-\frac{15x^2}{2}(1-1/\rho^2)\geq0
\]
Instead, when $\rho\geq1$ one immediately notices that $u_5$ is the fundamental solution of the Grushin sub-Laplacian with pole at the origin {(cf. \cite[Theorem 3.1 or Corollary 5.1]{BieskeGong}, being $Q=3$, see also \cite[Appendix B]{DaLucente} for the multi-dimensional counterpart)}. 
Similarly, $v_5=-u_5$ shows the failure of the Liouville property for subsolutions. 

This example underlines that as soon as the strict 
 ellipticity is not in force the Liouville property may fail even in the plane. 
 
 Similarly, following the above results for Carnot groups, it is immediate to provide a generalization of the previous counterexample to multi-dimensional Grushin geometries: for instance, the function $u(x)=(1+\rho^2)^{1-\frac{Q}{2}}$ is a nonnegative (and even bounded) classical supersolution to $-\Delta_\X u=0$ in $\R^d$, which shows the failure of the one-side Liouville property.

\smallskip
We now consider Liouville properties of the form \eqref{LP1}-\eqref{LP2} for viscosity sub- and supersolutions  of the quasi-linear equations
\begin{equation}\label{gruquasi1}
-\Delta_\X u+H_i(x,u,D_\X u)=0\text{ in }\R^d=\R^n\times\R^k
\end{equation}
and 
\begin{equation}\label{gruquasi2}
-\Delta_\X u+H_s(x,u,D_\X u)=0\text{ in }\R^d=\R^n\times\R^k\ ,
\end{equation}
where $\Delta_\X$ is the Grushin sub-Laplacian and $H_i$, $H_s$ are defined by \eqref{Hi}, \eqref{Hs} with $b^\alpha(x)$ taking values in $\R^{d}$. They are consequences of the general result in Theorem \ref{main} combined with Lemma \ref{hesgru} and can be proved exactly as in the case of Carnot structures, so we omit the proof.
\begin{thm}\label{Gruquasi}
Let $\mathcal{X}$ be the 
 vector fields in $\R^d=\R^n\times\R^k$ defined by \eqref{grushin_d}. Assume \eqref{b}-\eqref{c} and
\begin{equation}\label{Grucond}
\sup_{\alpha\in A}\left\{ b^\alpha(x)\cdot q-\frac{c^\alpha(x)}{|x|^{2\gamma}}\rho^{2\gamma+2}\log\rho\right\}\leq -(Q-2)
\end{equation}
for $\rho$ sufficiently large
, where $Q=n+(1+\gamma)k$ 
 and $q=\left(x,(1+\gamma)\frac{y}{|x|^\gamma}\right)\in\R^d=\R^n\times\R^k$.
\begin{itemize}
\item[(A)] If either $c^\alpha(x)\equiv0$ or $u\geq0$, then \eqref{LP1} holds for \eqref{gruquasi1} with $w(x)=\log\rho(x)$.
\item[(B)] If either $c^\alpha(x)\equiv0$ or $v\leq0$, then \eqref{LP2} holds for \eqref{gruquasi2} with $W(x)=-\log\rho(x)$.
\end{itemize}
\end{thm}

\subsection{Liouville theorems for fully nonlinear problems}
In the Grushin geometry there is no  fundamental solutions to the Pucci's extremal operators, as we had in Section \ref{sec_fund}. Then we can give a sharp result only for the fields \eqref{grushin} in the plane, whereas for the multi-dimensional case \eqref{grushin_d} we will give some sufficient conditions in two remarks at the end of the section.
We start by considering the following nonlinear equations 
\begin{equation}\label{pucci1G}
\mathcal{M}^-_{\lambda,\Lambda}((D^2_{\X}u)^*)+H_i(x,u,D_{\X}u)=0 \quad \text{ in }\R^2\ ,
\end{equation}
and
\begin{equation}\label{pucci2G}
\mathcal{M}^+_{\lambda,\Lambda}((D^2_{\X}u)^*)+H_s(x,u,D_{\X}u)=0 \quad \text{ in }\R^2\ .
\end{equation}
\begin{thm}\label{main;gru}
 Let $\X$ and $\rho$ be defined by \eqref{grushin} and \eqref{homnorm}. Assume \eqref{b}-\eqref{c} and
\begin{equation}\label{condcor1grushin}
2\sup_{\alpha\in A}\{b^\alpha(x)\cdot \tilde{\eta}-c^\alpha(x)\rho^4\log\rho\}\leq(-\Lambda-\lambda)x^2+(-\Lambda+\lambda)\sqrt{9x^4+4y^2}\ ,
\end{equation}
for $|x|,|y|$ sufficiently large, where $\tilde{\eta}:=(x^3,2xy)\in\R^2$ and $x\neq 0$.
\begin{itemize}
\item[(a)] If either $c^\alpha(x)\equiv0$ or $u\geq0$, then \eqref{LP1} holds for \eqref{pucci1G} with $w(x)=\log\rho(x)$.
\item[(b)] If either $c^\alpha(x)\equiv0$ or $v\leq0$, then \eqref{LP2} holds for \eqref{pucci2G} $W(x)=-\log\rho(x)$.
\end{itemize}
\end{thm}
\begin{proof}
As in the proofs of Theorems  \ref{Cor1H} and 
 \ref{corfree} we compute the symmetrized horizontal Hessian of the Lyapunov function $w=\log\rho$. We first note that $w(x)$ explodes as $\rho\to\infty$. 
By Lemma \ref{hessianradial} we have $D_\X\rho= 
\tilde{\eta}/\rho^3$ and the formula 
 for the symmetrized horizontal Hessian of $w$
\begin{equation*}
(D^2_{\mathcal{X}}w)^*=\frac{1}{\rho^4}\begin{pmatrix}3x^2& y\\y & 2x^2\end{pmatrix}-\frac{4}{\rho^2}D_{\mathcal{X}}\rho\otimes D_{\mathcal{X}}\rho\ .
\end{equation*}
We claim that the eigenvalues are
\begin{equation}\label{eiggru}
\left\{\frac{x^2+\sqrt{9x^4+4y^2}}{2\rho^4},\frac{x^2-\sqrt{9x^4+4y^2}}{2\rho^4}\right\}\ .
\end{equation}
Indeed the (symmetrized) horizontal Hessian is given by
\begin{equation*}
(D^2_{\mathcal{X}}w)^*=\begin{pmatrix}\frac{3x^2}{\rho^4}-\frac{4x^6}{\rho^8}& \frac{y}{\rho^4}-\frac{8x^4y}{\rho^8}\\\frac{y}{\rho^4}-\frac{8x^4y}{\rho^8} & \frac{2x^2}{\rho^4}-\frac{16x^2y^2}{\rho^8}\end{pmatrix}\ .
\end{equation*}
Then one computes
\begin{equation*}
\mathrm{Tr}((D^2_{\mathcal{X}}w)^*)=\frac{5x^2}{\rho^4}-\frac{4x^2}{\rho^8}(x^4+4y^2)=\frac{x^2}{\rho^4}\ ,
\end{equation*}
and, by recalling the expression of $\rho$, we also get
\begin{equation*}
\det((D^2_{\mathcal{X}}w)^*)=\frac{6x^4}{\rho^8}-\frac{48x^4y^2}{\rho^{12}}-\frac{8x^8}{\rho^{12}}+\frac{64x^8y^2}{\rho^{16}}-\frac{y^2}{\rho^8}-\frac{64x^8y^2}{\rho^{16}}+\frac{16x^4y^2}{\rho^{12}}=
\end{equation*}
\begin{equation*}
=\frac{(6x^4-y^2)}{\rho^8}-\frac{32x^4y^2}{\rho^{12}}-\frac{8x^8}{\rho^{12}}{=\frac{-2x^4-y^2}{\rho^8}}\ .
\end{equation*}
Then the eigenvalues are given by the formulae 
\begin{equation*}
\lambda_1:=\frac{\mathrm{Tr}((D^2_{\mathcal{X}}w)^*) {+} \sqrt{\mathrm{Tr}((D^2_{\mathcal{X}}w)^*)^2-4\det((D^2_{\mathcal{X}}w)^*)}}{2},
\end{equation*}
and
\begin{equation*}
\lambda_2:=\frac{\mathrm{Tr}((D^2_{\mathcal{X}}w)^*)  {-} \sqrt{\mathrm{Tr}((D^2_{\mathcal{X}}w)^*)^2-4\det((D^2_{\mathcal{X}}w)^*)}}{2}\ .
\end{equation*}
Note that
\[
\sqrt{\mathrm{Tr}((D^2_{\mathcal{X}}w)^*)^2-4\det((D^2_{\mathcal{X}}w)^*)}=\frac{1}{\rho^4}\sqrt{9x^4+4y^2}\ .
\]
Then, we get the eigenvalues \eqref{eiggru}. In particular, we immediately observe that $\lambda_1$ is positive and $\lambda_2$ is negative and this fact allows to compute Pucci's extremal operators over $(D^2_{\mathcal{X}}w)^*$. We have
\begin{equation*}
\mathcal{M}^-_{\lambda,\Lambda}((D^2_{\mathcal{X}}w)^*)+\inf_{\alpha\in A}\{c^\alpha(x)\log\rho-b^\alpha(x)\cdot \frac{\tilde \eta}{\rho^4}\} 
\end{equation*}
\begin{equation*}
=-\Lambda\frac{x^2+\sqrt{9x^4+4y^2}}{2\rho^4}-\lambda\frac{x^2-\sqrt{9x^4+4y^2}}{2\rho^4}
+\inf_{\alpha\in A}\{c^\alpha(x)\log\rho-b^\alpha(x)\cdot \frac{\tilde \eta}{\rho^4}\}\geq0
\end{equation*}
if condition \eqref{condcor1grushin} is satisfied. One can obtain the same sufficient condition for equations of the form \eqref{pucci2G} using the Lyapunov function $W(\rho)=-\log\rho$.
\end{proof}

\begin{rem}
\label{high}
Sufficient conditions similar to those in Theorem \ref{main;gru} can be obtained in the same way for the multi-dimensional case by taking again $\log\rho$ 
as Lyapunov function and using Lemma \ref{hesgru}. However, they are not optimal because the Pucci's extremal operator cannot be computed explicitly if $d\geq 3$ and must be estimated, as in Theorem \ref{corfree}.
\end{rem}
\begin{rem}
\label{high2}
As in Section \ref{sec_cc} we can use the result for the quasilinear case, Theorem \ref{Gruquasi}, for equations of the form
\begin{equation*}
F((D^2_\X u)^*)+H_i(x,u,D_\X u)=0 \quad\text{ in }\R^d\ 
\end{equation*} 
{for any $d\geq 2$, }if for some $\lambda>0$ $F(M)\geq -\lambda\mathrm{Tr}(M)$ for all $M\in \Sym_{{m}}$.
In fact  a subsolution $u$ is also  subsolution of \eqref{gruquasi1}, so under assumption \eqref{Grucond} it satisfies the statement (A) of Theorem \ref{Gruquasi}. This result applies to $F=\mathcal{M}^+_{\lambda,\Lambda}$, and a symmetric one holds for supersolutions of $\mathcal{M}^-_{\lambda,\Lambda}((D^2_\X u)^*)+H_s(x,u,D_\X u)=0$ 
 as in Remark \ref{conv-conc}.
\end{rem}

\begin{rem}
A Liouville comparison principle like Corollary \ref{comparisonH} 
holds for equations \eqref{gruquasi1} and \eqref{gruquasi2} assuming \eqref{Grucond} instead of \eqref{condH}, and for equations \eqref{pucci1G} and \eqref{pucci2G} assuming \eqref{condcor1grushin}.
\end{rem}
\begin{rem}
Results for equations involving the horizontal Hessian $D^2_\X u$ and the Euclidean gradient $Du$ can be found as in Section \ref{subsec_mix}.
\end{rem}


\section{Some intermediate structures: the Heisenberg-Greiner vector fields}
\label{sec_hg}
\subsection{Basic properties} In this section we consider a sub-Riemannian structure which can be seen as intermediate among the Heisenberg group and Grushin geometries. We consider $(x,t)\in\R^{2d}\times\R$, $r=|x|$, $\delta\geq1$ an integer, and the H-G vector fields
\[
X_i=\partial_{x_i}+2\delta x_{i+d}r^{2\delta-2}\partial_t\,, \quad
X_{i+d}=\partial_{x_{i+d}}-2\delta x_{i}r^{2\delta-2}\partial_t\,, \quad i=1,...,d .
\]
We set $Q:=2d+2\delta$ the homogeneous dimension. These correspond to the Heisenberg vector fields when $\delta=1$, while for $\delta>1$ and integer these are called Greiner vector fields, cf \cite{Gcan}. Furthermore, by letting $\delta\to0$ one obtains the Euclidean fields in $\R^{2d}$, while the presence of the coefficient $r$ in the term involving $\partial_t$ reminds the Grushin-type fields already discussed.  Here, it is natural to consider the gauge norm
\[
N((x,y))=((x_1^2+....+x_{2d}^2)^{2\delta}+t^2)^{\frac{1}{4\delta}}=(r^{4\delta}+t^2)^{\frac{1}{4\delta}}.
\]
It is proved in \cite{BGG1,BGG2,BieskeF} that $\Gamma=N^{2-Q}$ is the fundamental solution to the Heisenberg-Greiner operator $\Delta_\X=\sum_{i=1}^{2d}X_i^2$. Then $N$ plays exactly the same role as $\rho$ in the previous sections, we change notation only for consistency with the cited literature.
We now compute the $\X$ derivatives of $N$. 
\begin{lemma}\label{hesGre}
For the vector fields just defined we have
\[
D_\X N=\frac{\eta}{N^{4\delta-1}}\,, \quad |D_\X N|=\frac{r^{2\delta-1}}{N^{2\delta-1}} \,,
\]
for $\eta\in\R^{2d}$ defined as 
$\eta_i=x_ir^{4\delta-2}+x_{i+d}r^{2\delta-2}t$,
$\eta_{i+d}=x_{i+d}r^{4\delta-2}-x_{i}r^{2\delta-2}t$, 
and 
\[
\Delta_\X N=|D_\X N|^2\left(f''(N)+\frac{Q-1}{N}f'(N)\right)\ .
\]
\end{lemma}
\begin{proof} We compute 
\[
X_i N=\frac{x_ir^{4\delta-2}}{N^{4\delta-1}}+\frac{x_{i+d}r^{2\delta-2}t}{N^{4\delta-1}} \,,\quad
X_{i+d} N=\frac{x_{i+d}r^{4\delta-2}}{N^{4\delta-1}}-\frac{x_{i}r^{2\delta-2}t}{N^{4\delta-1}} \,.
\]
Therefore
\begin{equation}\label{grad}
|D_\X N|^2=\frac{r^{8\delta-2}}{N^{2(4\delta-1)}}+\frac{r^{4\delta-2}t^2}{N^{2(4\delta-1)}}=\frac{r^{4\delta-2}}{N^{2(4\delta-1)}}\underbrace{(r^{4\delta}+t^2)}_{N^{4\delta}}=\frac{r^{2(2\delta-1)}}{N^{2(2\delta-1)}}
\end{equation}
Moreover, for $i=1,...,d$ we have
\[
X_i^2N=\frac{r^{4\delta-2}}{N^{4\delta-1}}+(4\delta-2)\frac{r^{4(\delta-1)}x_{i}^2}{N^{4\delta-1}}+2x_{i+d}x_i(\delta-1)\frac{r^{2\delta-4}t}{N^{4\delta-1}}+2\delta x_{i+d}^2\frac{r^{4\delta-4}}{N^{4\delta-1}}-\frac{4\delta-1}{N}X_iNX_iN
\]
\[
X_{i+d}^2N=\frac{r^{4\delta-2}}{N^{4\delta-1}}+(4\delta-2)\frac{r^{4(\delta-1)}x_{i+d}^2}{N^{4\delta-1}}-2x_{i+d}x_i(\delta-1)\frac{r^{2\delta-4}t}{N^{4\delta-1}}+2\delta x_{i}^2\frac{r^{4\delta-4}}{N^{4\delta-1}}-\frac{4\delta-1}{N}X_{i+d}NX_{i+d}N
\]
Then, using that $Q=2d+2\delta$ and \eqref{grad}, we compute
\[
\Delta_\X f(N)=f'(N)\Delta_\X N+f''(N)|D_\X N|^2=|D_\X N|^2\left(f''(N)+\frac{Q-1}{N}f'(N)\right)\ .
\]
%
%
%

\end{proof}
\subsection{Some Liouville theorems for linear and quasi-linear problems}
\label{sec;HGlql}
Let $\X$ be the H-G 
 vector fields. We first 
 observe that the Liouville property for subsolutions (supersolutions) bounded from above (below) to $-\Delta_\X u=0$ in $\R^d$ fails also in this setting. Indeed, by means of 
 Lemma \ref{hesGre} one checks 
  that the function $u(x)=-(1+\rho^2)^{-\frac{Q-2}{2}}$ is a non-trivial bounded classical subsolution to the corresponding sub-Laplace equation in the whole space. The same function provides a counterexample for the validity of the Liouville property for sub-solutions bounded from above to $\mathcal{M}^-_{\lambda,\Lambda}((D^2_\X u)^*)=0$ in $\R^d$ via the inequality $\mathcal{M}^-_{\lambda,\Lambda}((D^2_\X u)^*)\leq -\Lambda \Delta_\X u$. 

Consider now
\begin{equation}
\label{sub1}
-\Delta_\X u+H_i(x,u,D_\X u)=0\text{ in }\R^{2d+1}
\end{equation}
and
\begin{equation}\label{sub2}
-\Delta_\X u+H_s(x,u,D_\X u)=0\text{ in }\R^{2d+1}
\end{equation}
where $\Delta_\X$ is the H-G 
 sub-Laplacian and the vector fields $b^\alpha(x)$ in $H_i$ and $H_s$ take values in $\R^{2d}$. 
 By the usual proof based on Theorem  \ref{main} and now combined with Lemma \ref{hesGre} we get the following.
\begin{thm}
\label{Cor1G}
Let $\mathcal{X}=\{X_1,....,X_{2d}\}$ be the H-G 
vector fields. Assume \eqref{b}-\eqref{c} 
and the condition
\begin{equation}\label{condcor1}
\sup_{\alpha\in A}\left\{ b^\alpha(x)\cdot \frac{\eta}{|x|^{2(2\delta-1)}}-\frac{c^\alpha(x)N^{4\delta}\log N}{|x|^{2(2\delta-1)}}\right\}\leq -(Q-2)
\end{equation}
for $N$ sufficiently large, $|x|\neq0$, where $Q=2d+2\delta$
, 
 and $\eta\in\R^{2d}$ 
is defined in Lemma \ref{hesGre}.
\begin{itemize}
\item[(A)] If either $c^\alpha(x)\equiv0$ or $u\geq0$, then \eqref{LP1} holds for \eqref{sub1} with $w(x)=\log N(x)$.
\item[(B)] If either $c^\alpha(x)\equiv0$ or $v\leq0$, then \eqref{LP2} holds for \eqref{sub2} with $W=-\log N(x)$.
\end{itemize}
\end{thm}


\section{Optimality of the conditions}
\label{sec_fin}
\begin{rem}
The sufficient condition 
{found in \cite[Remark 4.5]{BG_lio1}} for the Liouville property of the equation
\begin{equation}\label{M+b}
\mathcal{M}^+_{\lambda,\Lambda}(D^2_{\X}u)-b(x)\cdot D_\X u=0\text{ in }\R^{2d+1} \,,
\end{equation}
where $\X$ are the Heisenberg vector fields, is optimal. { Such assumptions can be rewritten, with the help of \cite[Lemma 4.1]{BG_lio1}, as}
\begin{equation}
\label{ass_heis}
\limsup_{\rho(x)\to\infty}\rho\ b(x)\cdot D_\X \rho< \lambda(2-\beta)|D_\X \rho|^2\ ,
\end{equation}
where $\beta=\frac{\Lambda}{\lambda}(Q-1)+1$ is the intrinsic dimension of the Pucci-Heisenberg maximal operator and $\rho(x)$ is the homogeneous norm.
For $\delta>0$ to be later determined we take, {as in \cite{BG_lio1,Goffi}}, the radial function $u(x)=f(\rho(x))=(1+\rho^2)^{-\frac{\delta}{2}}$. In view of \cite[Lemma 4.1]{BG_lio1}, we compute
\[
f'(\rho)=-\delta(1+\rho^2)^{-\frac{\delta}{2}-1}\rho\,, \quad
f''(\rho)=\delta(1+\rho^2)^{-\frac{\delta}{2}-2}\left[(\delta+1)\rho^2-1\right]\ ,
\]
so that the eigenvalues of $(D^2_\X f)^*$ are $f''(\rho)|D_\X \rho|^2$, $3\frac{f'(\rho)}{\rho}|D_\X \rho|^2$, which are simple, and $\frac{f'(\rho)}{\rho}|D_\X \rho|^2$ with multiplicity $Q-2$. 
Therefore, when $\rho^2\leq 1/(\delta+1)$ we have, using that $\Lambda(Q-1)=\lambda(\beta-1)$,
\begin{multline*}
\mathcal{M}^+_{\lambda,\Lambda}(D^2_{\X}u)=|D_\X \rho|^2\delta(1+\rho^2)^{-\frac{\delta}{2}-1}\left[\Lambda \frac{1-(\delta+1)\rho^2}{1+\rho^2}+\Lambda(Q-1)\right]\\\geq\lambda |D_\X \rho|^2\delta(1+\rho^2)^{-\frac{\delta}{2}-1} \left[\frac{1-(\delta+1)\rho^2}{1+\rho^2}+\beta-1\right]\ .
\end{multline*}
Similarly, when $\rho^2\geq 1/(\delta+1)$ it follows that
\begin{multline*}
\mathcal{M}^+_{\lambda,\Lambda}(D^2_{\X}u)=\lambda|D_\X \rho|^2\delta(1+\rho^2)^{-\frac{\delta}{2}-1}\left[ \frac{1-(\delta+1)\rho^2}{1+\rho^2}\right]+\Lambda(Q-1)\delta(1+\rho^2)^{-\frac{\delta}{2}-1}|D_\X \rho|^2\\
=\lambda |D_\X \rho|^2\delta(1+\rho^2)^{-\frac{\delta}{2}-1} \left[\frac{1-(\delta+1)\rho^2}{1+\rho^2}+\beta-1\right]\ .
\end{multline*}
In both cases we end up with
\[
\mathcal{M}^+_{\lambda,\Lambda}(D^2_{\X}u)\geq \lambda |D_\X \rho|^2\delta(1+\rho^2)^{-\frac{\delta}{2}-1} \left[\frac{1-(\delta+1)\rho^2}{1+\rho^2}+\beta-1\right]\ .
\]
We then take the vector field $b:\R^{2d+1}\to\R^{2d}$, $b(x)=\lambda(2-\beta+\delta)\frac{\rho}{1+\rho^2}D_\X \rho$ to find that
\[
\mathcal{M}^+_{\lambda,\Lambda}(D^2_{\X}u)-b(x)\cdot D_\X u\geq \beta\lambda\delta{ |D_\X \rho|^2}(1+\rho^2)^{-\frac{\delta}{2}-2}
\]
so that 
 $u$ is a bounded non-constant supersolution to the equation. 
On the other hand 
\[
\lim_{\rho\to\infty}\rho\ b(x)\cdot D_\X \rho = 
\lambda(2-\beta+\delta)|D_\X\rho|^2
\]
and $\delta$ can be taken arbitrarily small, so the condition \eqref{ass_heis} is sharp.

The same counterexample works with minor modifications for the H-type group considered in Section \ref{sec_fund}, replacing $Q$ with the appropriate homogeneous dimension. 

Moreover, by setting $\lambda=\Lambda=1$, the same radial function works as a counterexample for the linear equation
\[
-\Delta_\X u-b(x)\cdot D_\X u=0\text{ in }\R^d\ ,
\]
provided that the corresponding sub-Laplacian admits a fundamental solution. This is true on any Carnot group with step 2, on Grushin geometries, and for problems structured on Heisenberg-Greiner vector fields. The Euclidean counterpart of such counterexamples can be found in \cite{CirantGoffi}. This also shows that the condition ensuring the Liouville comparison principle in Corollary \ref{lcpl} is sharp. 

This agrees with the analysis on Riemannian manifolds carried out in \cite{MV} (see also the references therein), where the Liouville property has been proved to be equivalent to the existence of a Lyapunov function (named Khas'minskii test) for equations driven by $p$-Laplacians perturbed by zero-th order terms. Our results show analogous properties for some fully nonlinear PDEs on special sub-Riemannian geometries of Heisenberg type, e.g.,  for the equations driven by Pucci's extremal operators over the horizontal Hessian 
 perturbed with drift terms. 
 We do not study here the connection of the Liouville property for uniformly elliptic or 
 degenerate Bellman-Isaacs equations with the probabilistic properties of the corresponding controlled diffusion 
  processes,  that in the case of linear equations and uncontrolled processes was deeply studied, e.g., in \cite{Gry1}.
\end{rem}
\begin{rem}
Regarding equation \eqref{M+b}, our sufficient conditions give new Liouville properties even for problems perturbed with first order terms having natural gradient growth. For instance, the Liouville property for equation \eqref{M+b} leads to a nonexistence result for supersolutions bounded below of an equation like
\[
\mathcal{M}^+_{\lambda,\Lambda}(D^2_{\X}u)\pm |D_\X u|^2-b(x)\cdot D_\X u=0\text{ in }\R^{2d+1}\ .
\]
When the eikonal term has a minus sign, a supersolution to the above equation is also a supersolution to \eqref{M+b}. In the case of the presence of a positive quadratic term, one can use a nonlinear variant of the Hopf-Cole transformation as $v(x)=\lambda(1-e^{-\frac{u}{\lambda}})$ to reduce the proof of the Liouville property to an equation in the variable $v$ with only a drift term of Ornstein-Uhlenbeck type. One can even consider a similar problem where the nonlinearity involves the product $u^q|Du|^2$, $q\geq0$, where still the first-order superlinear term has a quadratic growth. 
We refer to \cite{CirantGoffi} for further details and open problems involving such nonlinearities.
\end{rem}
%
{\section{Strict Lyapunov functions and ergodicity}
\label{sec;erg}
In this section we briefly discuss how the Liouville properties for linear degenerate elliptic operators allow to address the question of the ergodicity of stochastic processes with degenerate diffusion in the whole space $\R^d$ with possibly unbounded coefficients, together with the large time behavior for parabolic problems posed on the whole space. Although the existence of a Lyapunov/exhaustion function is in general sufficient to derive the Liouville property for linear elliptic equations with drift terms, it does not guarantee the existence of an invariant probability measure for the underlying diffusion process, cf \cite[Remark (iii), p. 9]{LionsMusiela}, see also \cite{AB} for the relevant definitions. 

{When the diffusion term is strictly elliptic, it is well-known that a sufficient condition for the existence of an invariant measure for a process with infinitesimal generator $\mathcal{L}$ is the existence of a \textit{strict Lyapunov function}, namely the existence of a function $\overline{w}\in C^\infty$ such that 
\[
\text{$\overline{w}\to+\infty$ as $|x|\to\infty$ and $\mathcal{L}\overline{w}\geq1$ for $|x|\geq R_0>0$}.
\]
We consider the example of the Heisenberg group in $\R^{2d+1}$ treated in \cite{BG_lio1}, but general results on different structures can be obtained in the same manner. We denote for $x_H:=(x,y)\in\R^{2d}$ by
\begin{equation}
\label{sigma}
\sigma=
\begin{pmatrix}
I_d & 0_d\\
0_d & I_d\\
2y & -2x
\end{pmatrix}
\end{equation}
and consider the operator (written in Euclidean coordinates) 
$$\mathcal{L}u:=-\mathrm{Tr}(\sigma\sigma^TD^2u)-b(x)\cdot Du.$$
 Since the matrix $A=\sigma\sigma^T$ is degenerate, to use the previous sufficient condition valid for a strictly elliptic operator it is enough to regularize the structure by considering the approximated operator $\mathcal{L}_\eps u=-\mathrm{Tr}(A_\eps(x)D^2u)-b(x)\cdot Du$, where
 \[
A_\eps(x)=\sigma\sigma^T+
\begin{pmatrix}
0_d & 0_d & 0\\
0_d & 0_d & 0\\
0 & 0& \eps^2
\end{pmatrix}=
\begin{pmatrix}
I_d & 0_d & 2y\\
0_d & I_d & -2x\\
2y & -2x& 4|x_H|^2+\eps^2
\end{pmatrix},
\]
 construct an appropriate Lyapunov function, and finally pass to the limit. To this aim, we show that a slight modification of Corollary 4.19 in \cite{BG_lio1}, or of Corollary \ref{cor_euc1} for H-type groups in the present paper, leads to the existence of a strict Lyapunov function for { $\mathcal{L}_\eps$, uniform in $\eps$.}
 %
\begin{lemma}
\label{strictlyap}
Consider $\mathcal{L}_\eps u=-\mathrm{Tr}(A_\eps(x)D^2u)-b(x)\cdot Du$ and suppose there exist $\gamma_1,...,\gamma_{2d+1}\in\R$ such that $\min \gamma_i=\gamma_0>0$ satisfying
\begin{equation}\label{erg}
b(x)\cdot D\rho\leq -\sum_{i=1}^{2d+1}\gamma_ix_i\partial_i\rho+o\left(\frac{1}{\rho^3}\right)\text{ as }\rho\to\infty.
\end{equation}
Then there { exist $\overline{w}$ and $R_0$ independent of $\eps\in[0,1]$ such that $\overline{w}$ is a strict Lyapunov function for 
 $\mathcal{L}_\eps$.} 
\end{lemma}
\begin{proof}
We take $\overline{w}=\nu\log\rho$, where $\rho=(|x_H|^4+x_{2d+1}^2)^\frac14$ and $\nu>0$ to be later determined. Then
\[
\mathcal{L}_\eps\overline{w}=-{ 2d}
\nu\frac{|x_H|^2}{\rho^4}-\frac{\nu\eps^2}{2\rho^4}+{\nu}{ \frac{\eps^2x_{2d+1}^2}{\rho^8}}-\nu b(x)\cdot \frac{D\rho}{\rho}\ .
\]
Since $D\rho=(2|x_H|^2x_H,x_{2d+1})/(2\rho^3)$, we get from \eqref{erg} that 
\begin{multline*}
\mathcal{L}_\eps\overline{w}\geq -{ 2d}\nu
\frac{|x_H|^2}{\rho^4}-\frac{\nu\eps^2}{2\rho^4}+\frac{\nu}{2\rho^4}\left(2\sum_{i=1}^{2d}\gamma_ix_i^2|x_H|^2+\gamma_{2d+1}x_{2d+1}^2+o(1)\right)\\
\geq \nu\frac{\gamma_0}{2}+\frac{\nu}{\rho^4}\left[|x_H|^2(\frac{\gamma_0}{2}|x_H|^2-{ 2d})
-\frac{\eps^2}{2}+o(1)\right].
\end{multline*}
Now we choose $\nu= 4/\gamma_0$ and observe that $|x_H|^2(\frac{\gamma_0}{2}|x_H|^2-{ 2d})$ is bounded from below. Then we can choose $R$ independent of $\eps\in[0,1]$ such that
$
\mathcal{L}_\eps\overline{w}\geq1\text{ for }\rho\geq R ,
$ 
which easily gives the claim. 
\end{proof}
{Thanks to the previous Lemma one can follow the approach initiated in \cite{LionsCourse} and outlined in 
 \cite[Prop. 2.1]{MMT} to obtain in a rather straightforward way the following result.
\begin{thm}
\label{thmerg}
Assume $\sigma$ is given by \eqref{sigma} and $b$ is Lipschitz and satisfies \eqref{erg}. Then
 there exists a unique invariant probability measure $m$ 
 for the diffusion process generated by the operator $\mathcal{L}$. 
\end{thm}
\begin{rem}
\label{exten}
It is not difficult to adapt the proof of Lemma \ref{strictlyap} 
to prove the existence of a strict Lyapunov function when the operator is modeled over the fields generating any Carnot group of step 2, such as a H-type group, or Grushin-type geometries. Therefore also Theorem \ref{thmerg} can be extended to these settings.
\end{rem}
\begin{rem}
\label{largetime}
The invariant measure $m$ found in Theorem \ref{thmerg}, together with the Liouville property and H\"older estimates for solutions of subelliptic equations, can be used to prove asymptotic properties of PDEs associated to the operator $\mathcal{L}$. A first example is the small discount limit for the stationary equation, for $\delta>0$,
\[
\delta u_\delta + \mathcal{L}u = f(x) ,\quad \text{in } \R^{2d+1} ,
\]
with $f$ continuous and bounded.
Under the assumptions of Theorem \ref{thmerg} or Remark \ref{exten} one can prove
\[
\lim_{\delta\to 0}u_\delta(x)=\int_{\R^{2d+1}}f(x)\,dm(x) ,
\]
uniformly on compact sets, see \cite[Theorem 4.3]{BCManca} or  \cite[Theorem 4.1]{MMT}.

A second example is 
 the large-time behavior of the  degenerate parabolic equation
\[
\begin{cases}
\partial_t u+\mathcal{L}u=0&\text{ in }\R^{2d+1}\times(0,\infty),\\
u(x,0)=f(x)&\text{ in }\R^{2d+1} . 
\end{cases}
\]
Then, under the same assumptions, one can prove
\[
\lim_{t\to+\infty}u(x,t)=\int_{\R^{2d+1}}f(x)\,dm (x),
\]
uniformly on compact sets, see \cite[Prop. 4.4]{BCManca} or  \cite[Theorem 4.2]{MMT}.
\end{rem}
}

\end{document}